\pdfoutput=1
\documentclass{amsart}
\usepackage{todonotes}
\usepackage{graphicx, amsmath, amssymb, amsthm}
\usepackage{tikz-cd}
\usepackage{spectralsequences}
\usepackage{hyperref} 
\usepackage[T1]{fontenc}
\hypersetup{
  colorlinks   = true, 
  urlcolor     = blue, 
  linkcolor    = blue, 
  citecolor   = red 
}
\usepackage{adjustbox}
\usepackage[backend=biber, style=numeric, maxbibnames=99, doi=true,url=false,
    giveninits=true, hyperref]{biblatex}
\renewbibmacro{in:}{}

\newcommand{\ip}[1]{\langle#1\rangle}
\newcommand{\g}{\gamma}

\newcommand{\fp}[1]{%
  \mathrm{fp}_{\le#1}%
}
\newcommand{\bprn}[1]{%
  BP_\R\ip{#1}%
}
\newcommand{\bpcm}[1]{BP^{(\!(C_4)\!)}\langle#1\rangle}

\newcommand{\bpgm}[1]{BP^{(\!(G)\!)}\langle#1\rangle}

\newcommand{\Z}{{\mathbb  Z}}

\newcommand{\R}{{\mathbb R}}
\newcommand{\F}{{\mathbb F}}

\newcommand{\BP}{BP}
\newcommand{\BPR}{\BP_{\R}}

\DeclareMathOperator{\Map}{Map}

\DeclareMathOperator{\Spec}{Spec}

\DeclareMathOperator{\Ind}{Ind}

\newcommand{\Sp}{\mathcal Sp}

\mathchardef\mhyphen=45

\numberwithin{equation}{section}

\newtheorem{theorem}{Theorem}[section]
\newtheorem{lemma}[theorem]{Lemma}
\newtheorem{corollary}[theorem]{Corollary}

\newtheorem{proposition}[theorem]{Proposition}

\newtheorem*{theorem*}{Theorem}
\newtheorem*{proposition*}{Proposition}

\theoremstyle{remark}
\newtheorem{remark}[theorem]{Remark}
\newtheorem{example}[theorem]{Example}
\newtheorem{notation}[theorem]{Notation}

\theoremstyle{definition}

\newtheorem{definition}[theorem]{Definition}

\begin{document}

\title{On higher real $K$-theories and finite spectra}
\author{Christian Carrick}
\address{Mathematical Institute, Utrecht University, Utrecht, 3584 CD, the Netherlands}
\email{c.d.carrick@uu.nl}
\thanks{The first author was supported by the National Science Foundation under Grant No. DMS-2401918 as well as the NWO grant VI.Vidi.193.111. The second author was supported by the National Science Foundation under Grant No. DMS-2105019}
\author{Michael A. Hill}
\address{University of Minnesota, Minneapolis, MN 55455}
\email{mahill@umn.edu}

\begin{abstract}
    We study higher chromatic height analogues $eo_h$ of the connective real $K$-theory spectrum $ko$. We show that $eo_h$ is an \emph{fp} spectrum of type $h$ in the sense of Mahowald--Rezk. We use these to study an Euler characteristic for \emph{fp} spectra introduced by Ishan Levy, and give a partial answer to a question of Levy regarding the algebraic $K$-theory of the category of finite type $h$ spectra. As a corollary, we prove that if the generalized Moore spectrum $\mathbb{S}/(2^{i_0},v_1^{i_1},\ldots,v_h^{i_h})$ exists, then the $2$-adic valuation of $\prod i_j$ must exceed that of the height $h$.
\end{abstract}

\maketitle
\tableofcontents
\section{Introduction}
Chromatic homotopy theory stratifies the stable homotopy category according to \emph{chromatic height} by pulling back the height filtration on the moduli of formal groups, via Quillen's theorem on complex bordism \cite{quillen}. Conceptually, the thick subcategory theorem of Hopkins--Smith \cite{hopkinssmith} states that every thick subcategory of finite $p$-local spectra is of the form
\[\Sp_{>h}^\omega:=\{X\text{ a finite }p\text{-local spectrum}\mid K(h)_*X=0\},\]
where $K(h)$ is the $h$-th Morava $K$-theory. Computationally, the chromatic spectral sequence of Miller--Ravenel--Wilson \cite{mrw} organizes the stable homotopy groups of spheres $\pi_*\mathbb{S}$ into $v_h$-periodic families.

Producing explicit information -- such as a particular type $>h$ finite spectrum or $v_h$-periodic family -- is extremely difficult. Essentially all of the progress in this direction has been done at small heights $h$ and with the help of designer cohomology theories like connective real $K$-theory $ko$ and topological modular forms $tmf$; see \cite{adamsperiodicity,mmmm,marksatya} for example. These are connective spectra with strong finiteness properties that have a good balance of richness and tractability. They are the prototypical examples of \emph{fp} spectra in the sense of Mahowald--Rezk \cite{MR}: bounded below, $p$-complete spectra with finitely presented cohomology over the Steenrod algebra.

At the prime $p=2$ and heights $h>2$, there is essentially no known explicit information about $v_h$-self maps and $v_h$-periodic families. There are also very few known examples of \emph{fp} spectra at these heights. We introduce a class of \emph{fp} spectra at each height $h$ that play an analogous role to $ko$ and $tmf$. The examples $ko$ and $tmf$ are sometimes referred to as $eo_1$ and $eo_2$, as they are connective models of the fixed points $EO_1$ and $EO_2$ of Lubin--Tate theories with respect to the maximal finite subgroup of the Morava stabilizer group $\mathbb{S}_h$ at heights $h=1$ and $2$, respectively.

At each height $h=2^{n-1}m$, there is a $C_{2^n}$ subgroup of $\mathbb{S}_h$, which is a maximal finite $2$-subgroup when $h\not\equiv 2$ mod $4$, by work of Hewett \cite{hewett}. In their work on the Kervaire invariant, the second author, Hopkins, and Ravenel introduced analogous theories $eo_h$ at these heights, coming from $C_{2^n}$-equivariant stable homotopy \cite{HHR}. These \emph{connective higher real $K$-theories} are also known as the $\bpgm{m}$'s (see Definition \ref{def:bpgm}), and they are formed by analogy with Atiyah's Real $K$-theory $k_\R$, which exhibits $ko$ as the fixed points of $ku$ with respect to its natural $C_2$-action \cite{atiyah}.

\begin{theorem}\label{thm:introthmeoh}
    The connective higher real $K$-theories $eo_h$ are fp spectra of type $h$.
\end{theorem}

The theories $eo_h$ provide a tractable way to access higher height phenomena, and Theorem \ref{thm:introthmeoh} establishes the desired finiteness properties that guarantee this tractability. These theories have been closely studied at heights $\le 4$ via the equivariant slice spectral sequence \cite{HHR,hhrc4,hswx}. For example, the height $4$ theory $eo_4$ is a connective model of the detection spectrum $\Omega$ of the second author, Hopkins, and Ravenel used in their solution of the Kervaire invariant one problem \cite{HHR}.

We give an application of Theorem \ref{thm:introthmeoh} at all heights to the nonexistence of generalized Moore spectra. In \cite{adamsjx4}, Adams constructed a self map of the mod $2$ Moore spectrum $\mathbb{S}/2$ inducing multiplication by $v_1^4$ on $BP$-homology. The periodicity theorem of Devinatz--Hopkins--Smith \cite{dhs} guarantees that at each height $h$, there exists an exponent $i_h>0$ and a self map $v_h^{i_h}$ of the generalized Moore spectrum $\mathbb{S}/(2^{i_0},v_1^{i_1},\ldots,v_{h-1}^{i_{h-1}})$ (when the latter exists) inducing multiplication by $v_h^{i_h}$ on $BP$-homology. The exponents that appear determine periodicities that appear in $\pi_*\mathbb{S}$, such as the $|v_1^4|=8$ fold periodicity observed by Adams in the image of the $J$-homomorphism.

Very little is known about which exponents $i_j$ can or cannot appear. At height $2$ for example, work of Behrens, the second author, Hopkins, and Mahowald \cite{mmmm} constructs a minimal $v_2^{32}$-self map on $\mathbb{S}/(2,v_1^4)$, and work of Behrens--Mahowald--Quigley \cite{bmq} constructs a minimal $v_2^{32}$-self map of $\mathbb{S}/(8,v_1^8)$. At odd primes $p$, work of Nave uses higher real $K$-theories to prove that the Smith--Toda complex $\mathbb{S}/(p,v_1,\ldots,v_h)$ does not exist if $h\ge (p+1)/2$ \cite{nave}.

We give the first nonexistence criterion at the prime $2$ valid at all heights. In the following, let $\nu_2(-)$ denote the $2$-adic valuation of an integer.

\begin{theorem}\label{thm:introthmmoore}
    If the generalized Moore spectrum $\mathbb{S}/(2^{i_0},\ldots,v_{h}^{i_{h}})$ exists, then 
    \[
        \nu_2\bigg(\prod\limits_{k=0}^hi_k\bigg)>\nu_2(h).
    \]
\end{theorem}

Theorem \ref{thm:introthmmoore} comes by way of an Euler characteristic on \emph{fp} spectra introduced by Ishan Levy \cite{levy}. Levy studies the algebraic $K$-theory $K_0(\Sp^\omega_{>h})$ of the category of type $>h$ finite spectra. Whereas the thick subcategory theorem classifies thick subcategories of finite spectra, information about $K_0(\Sp^\omega_{>h})$ may be used to classify \emph{stable} subcategories. Levy's Euler characteristic provides a map of abelian groups
\begin{equation}\label{eq:eulerbph}
    \chi_{BP\langle h\rangle}:K_0(\Sp_{>h}^\omega)\to\Z.
\end{equation}
Burklund--Levy prove that $\chi_{BP\langle h\rangle}$ is an isomorphism rationally \cite{burklundlevy}, and Levy conjectures that the image of $\chi_{BP\langle h\rangle}$ is generated by the order of the maximal finite $2$-subgroup of $\mathbb{S}_h$ (\cite[Question 7.4]{levy}). We partially verify Levy's conjecture at all heights $h\not\equiv 2$ mod $4$ by proving the following, of which Theorem \ref{thm:introthmmoore} is a corollary. 

\begin{theorem}\label{thm:introthmchi}
    Let $h=2^{n-1}m$. The map $\chi_{BP\langle h\rangle}$ of (\ref{eq:eulerbph}) is divisible by $2^n$.
\end{theorem}

When $h\not\equiv 2$ mod $4$, letting $m$ be odd, Levy's conjecture states that the image of $\chi_{BP\langle h\rangle}$ is generated by $2^n$. Theorem \ref{thm:introthmchi} implies that the generator is \emph{divisible} by $2^n$. This result constrains more generally the existence of finite type $h+1$ spectra that are in some sense \emph{close} to being a Smith--Toda complex $\mathbb{S}/(2,v_1,\ldots,v_h)$, like the type $2$ complex $Y/v_1=\mathbb{S}/(2,\eta,v_1)$ closely studied by Mahowald \cite{ehp} (see Remark \ref{remark:smallcomplexes}).

We prove the above theorems by a careful analysis of the formal group laws underlying the $eo_h$ theories, and by induction arguments on the height $h$. These induction arguments make use of the ``transchromatic'' phenomena present in the $eo_h$ theories that were observed by Meier--Shi--Zeng \cite{transchromatic} in the context of the equivariant slice spectral sequence. In fact, we show that certain height shifting layers in the slice filtration can be viewed as spectral analogues of Koszul resolutions. We implement these resolutions in a particularly convenient manner using the $G$-symmetric monoidal structure on the $\infty$-category of filtered $G$-spectra.

\subsection{Acknowledgments} The authors are grateful to Hausdorff Research Institute for Mathematics in Bonn funded by the Deutsche Forschungsgemeinschaft (DFG, German Research Foundation) under Germany's Excellence Strategy - EXC-2047/1 - 390685813 and to the organisers of the trimester program \emph{Spectral Methods in Algebra, Geometry, and Topology} 2022. The authors would like to thank the Isaac Newton Institute for Mathematical Sciences, Cambridge, for support and hospitality during the \emph{Equivariant Homotopy Theory in Context} 2025 programme where work on this paper was undertaken. This work was supported by EPSRC grant no EP/R014604/1.

The authors would also like to thank Max Blans, Lennart Meier, Ishan Levy, Jack Davies, Ryan Quinn, Sven van Nigtevecht, and Gijs Heuts for helpful comments and conversations. We thank the Online Encyclopedia of Integer Sequences (OEIS) for pointing us in the direction of the Gaussian binomial coefficients appearing in Section \ref{sec:algebraicstuff}.

\subsection{Notation and conventions}
\begin{enumerate}
    \item We work $2$-locally throughout this paper.
    \item Unless explicitly stated otherwise, $G$ will denote a cyclic group of order $2^n$, and $\gamma$ will denote a generator of $G$. We denote by $G'=C_{2^{n-1}}$ the index $2$ subgroup of $G$.
    \item We use the letter $h$ for chromatic heights, and when $m$ and $G=C_{2^n}$ are fixed, we set $h=m|G|/2=2^{n-1}m$.
    \item The notation $G\cdot x$ represents the set $\{x,\gamma x,\ldots,\gamma^{2^{n-1}-1}x\}$. For example, $\Z[G\cdot x]$ is the polynomial ring $\Z[x,\gamma x,\ldots,\gamma^{2^{n-1}-1}x]$. In the spectral context, this is understood in terms of twisted monoid rings, as in \cite[Section 2.4.1]{HHR}. For example, we set $\mathbb{S}[G\cdot \bar{x}]:=N_{C_2}^G(\mathbb{S}[\bar{x}])$ (see Definition \ref{def:equivariantcofiber}).
    \item We use both $\rho_G$ and $\rho_{|G|}$ to denote the regular representation of the group $G$.
    \item We use $\inf_{G/H}^G$ to denote the inflation functor, i.e. the left adjoint to the functor $\Sp^G\xrightarrow{(-)^H}\Sp^{G/H}$.
\end{enumerate}

\section{Outline of results and methods}

Mahowald and Rezk introduced the notion of an \emph{fp} spectrum, a \(p\)-complete, bounded below spectrum which has finitely presented cohomology as a module over the Steenrod algebra \cite{MR}. An equivalent condition is that a spectrum \(E\) is \emph{fp} if for any finite spectrum \(K\) of sufficiently large chromatic height, \(E\otimes K\) has finite homotopy. The smallest chromatic height for which \(E\otimes K\) has finite homotopy is the \emph{fp} type of \(E\).

Many kinds of spectra that arise from chromatic homotopy are \emph{fp} spectra:
\begin{itemize}
    \item \(H\mathbb F_p\) is \emph{fp} of type \(-1\)
    \item \(H\mathbb Z\) is \emph{fp} of type \(0\)
    \item both \(ku\) and \(ko\) are \emph{fp} of type \(1\)
    \item \(tmf\) and flavors like \(tmf_0(3)\) and \(tmf_1(3)\) are \emph{fp} of type \(2\)
    \item \(BP\ip{h}\) is \emph{fp} of type \(h\) for all \(h\).
\end{itemize}

Computationally, \emph{fp} spectra play a critical role in our understanding of stable homotopy groups of spheres, as seen especially in the example $tmf$ (see \cite{iwx} for example). They also play a central role in algebraic $K$-theory, as they appear in Rognes' redshift conjecture and in the work of Hahn--Wilson on redshift \cite{hahnwilson}.

\subsection{Higher truncated Brown--Peterson spectra and Koszul resolutions} Chromatic heights $1$ and $2$ are well served by the \emph{fp} spectra $ko$ and $tmf$. The spectrum $BP\ip{h}$ is not an appropriate replacement at higher heights because it is complex-orientable and thus has trivial Hurewicz image in positive degrees. One aim of this article is to introduce a class of interesting, workable examples of \emph{fp} spectra at all heights $h$ that approximate the sphere much better than $BP\ip{h}$.

One way of conceptualizing the sense in which $ko$ and $tmf$ approximate the sphere well at their respective heights is that they provide connective models of higher real $K$-theories. A theorem of Devinatz--Hopkins \cite{devhop} states that there is an equivalence $L_{K(h)}\mathbb{S}\simeq E_h^{h\mathbb{G}_h}$, where the lefthand side is the $K(h)$-local sphere spectrum and the right hand side is the fixed points of a Lubin--Tate theory of height $h$ with respect to the extended Morava stabilizer group $\mathbb{G}_h=\mathbb{S}_h\rtimes\mathrm{Gal}(\F_{p^h}/\F_p)$. Hopkins--Miller observed that by taking fixed points of $E_h$ with respect to \emph{finite} subgroups $G\subset\mathbb{G}_h$, the resulting \emph{higher real $K$-theories} $EO_h(G):=E_h^{hG}$ give tractable higher height analogues of real $K$ theory $KO$ \cite{hopmiller}. For example, at heights $h=1$ and $2$, the maximal finite subgroups of $\mathbb{S}_h$ are isomorphic to $C_2$ and   $G_{24}$ (the binary tetrahedral group) respectively. The fixed points $E_1^{hC_2}$ is the $2$-completion of $KO$, and $E_2^{hG_{24}}$ is closely related to $TMF$ (see \cite{ghmr}).

Working with the $EO_h(G)$'s in practice is limited by the size of these theories. The $EO_h(G)$'s are non-connective, and their mod $p$ homology vanishes, making it impossible to understand these theories directly from the point of view of the Adams spectral sequence. Moreover, the homotopy groups of the $EO_h(G)$'s are not degreewise finitely generated, so passing to the connective cover does not give a substantial improvement. A key feature of the study of the $EO_h(G)$'s is thus a search for good connective models $eo_h(G)$ with strong finiteness properties. These theories should, moreover, capture the same height $h$ information as $EO_h(G)$ in the sense that there is an equivalence 
\begin{equation}\label{eq:K(h)localization}
 L_{K(h)}eo_h(G)\simeq L_{K(h)}EO_h(G).   
\end{equation}

\subsubsection{Higher truncated Brown--Peterson spectra} At heights $h=2^{n-1}m$ with $p=2$, the Morava stabilizer group $\mathbb{S}_h$ contains a subgroup $G$ isomorphic to $C_{2^n}$, and every maximal $2$-subgroup of $\mathbb{S}_h$ is isomorphic to $C_{2^n}$ when $n\neq 2$ and $m$ is odd, that is when $h\not\equiv 2$ mod $4$. At each height $h$, the second author, Hopkins, and Ravenel  defined good candidates for connective models of $EO_h(H)$ via Real bordism, for $H$ any subgroup of $C_{2^n}$  \cite{HHR}. These theories are particularly accessible as they arise as the fixed points of a genuine $G$-spectrum---known as $\bpgm{m}$--- whose action comes from geometry, as opposed to the action on $E_h$, which is defined via obstruction theory. These are related to the $eo_h$-theories discussed in the introduction by the following definition.

\begin{definition}\label{def:eodef}
    Fix $m$ odd, set $h=2^{n-1}m$ and $G=C_{2^n}$, and let $H$ be a subgroup of $G$. We define $eo_h(H):=\bpgm{m}^H$, and we write $eo_h:=eo_h(G)=\bpgm{m}^G$ when $h\not\equiv 2$ mod $4$, since $G$ is a maximal $2$-subgroup of $\mathbb{S}_h$ in this case.
\end{definition}

The construction of $\bpgm{m}$ is very analogous to that of Atiyah's connective $K$-theory with Reality, $k_\R$ \cite{atiyah}, which is recovered by taking $G=C_2$ and $m=1$. Recall that the Brown--Peterson spectrum $BP$ has homotopy groups \[BP_*=\Z[v_1,v_2,\ldots],\]
and the truncated Brown--Peterson spectrum is defined as the quotient
\[BP\ip{h}:=BP/(v_{h+1},v_{h+2},\ldots),\]
so that
\[BP\ip{h}_*=\Z[v_1,v_2,\ldots,v_h].\]
The complex bordism spectrum $MU$ admits a $C_2$-action by complex conjugation, which underlies the genuine $C_2$-spectrum $MU_\R$, the Fujii--Landweber Real bordism spectrum \cite{fujii}\cite{landweber}. The Quillen idempotent lifts to an equivariant idempotent, forming the Real Brown--Peterson spectrum $BP_\R$.


The crucial algebraic input to the slice theorem of the second author, Hopkins, and Ravenel, proved in their work on the Kervaire invariant, was the observation that one may form the $G$-spectrum $MU^{(\!(G)\!)}:=N_{C_2}^GMU_\R$, and its $2$-local summand
\[BP^{(\!(G)\!)}:=N_{C_2}^{G}BP_\R,\]
 and the underlying homotopy groups of $MU^{(\!(G)\!)}$ and $BP^{(\!(G)\!)}$ admit explicit descriptions as $G$-algebras \cite[Proposition 5.45]{HHR}. In fact,
\[\pi_{*}^e BP^{(\!(G)\!)}=\Z[G\cdot t_1^G,G\cdot t_2^G,\ldots],\]
where $|t_i^G|=2(2^i-1)$. The notation $G\cdot t_i^G$ refers to the list of polynomial generators $\{t_i^G,\gamma t_i^G,\ldots,\gamma^{2^{n-1}-1}t_i^G\}$ where $\gamma$ is a generator of $G$, and $\gamma^{2^{n-1}}t_i^G=-t_i^G$. The underlying spectrum of $BP^{(\!(G)\!)}$ is the tensor power $BP^{\otimes 2^{n-1}}$, which is easily seen to have polynomial homotopy groups, but the $G$-algebra description above is much more subtle. Similar explicit descriptions are not known for non cyclic $2$-groups, e.g. $G=Q_8$.

The higher truncated Brown--Peterson spectra are now constructed analogously to the $BP\ip{h}$'s: one forms a quotient
\[\bpgm{m}:=BP^{(\!(G)\!)}/(G\cdot t_{m+1}^G,G\cdot t_{m+2}^G,\ldots)\]
so that 
\[\pi_*^e\bpgm{m}=\Z[G\cdot t_1^G,\ldots, G\cdot t_m^G].\]
One must form this quotient with care so as to obtain a $G$-equivariant quotient. This is done via the method of twisted monoid rings due to the second author, Hopkins, and Ravenel \cite[Section 2]{HHR}, which we recall in Section \ref{sec:filtration}.

\begin{remark}
    When $G=C_2$, the $C_2$-spectra $\bpgm{m}$ are also called $BP_\R\ip{m}$, or the Real truncated Brown--Peterson spectra, as they are $C_2$-equivariant lifts of $BP\ip{m}$.
\end{remark}

\begin{remark}
    At height $0$, one has $BP^{(\!(G)\!)}\langle 0\rangle=H\underline{\Z}$ for all $G$, and $H\underline{\Z}^H=H\Z$ for all $H\subset G$. At height $1$, Atiyah's $k_\R$ is a $C_2$-$\mathbb{E}_\infty$ form of the spectrum $BP^{(\!(C_2)\!)}\langle 1\rangle$ with underlying spectrum $ku$ and fixed point spectrum $ko$. At height $2$, a theorem of the second author and Meier \cite{hillmeier} shows that there is a $C_2$-$\mathbb{E}_\infty$ form of  $BP^{(\!(C_2)\!)}\langle 2\rangle$ with underlying spectrum $tmf_1(3)$ and fixed point spectrum $tmf_0(3)$. It is expected that there is a $C_4$-$\mathbb{E}_\infty$ form of  $BP^{(\!(C_4)\!)}\langle 1\rangle$ with underlying spectrum $tmf_1(5)$ and $C_4$ fixed points $tmf_0(5)$. Beyond these cases, it is not known if $\bpgm{m}$ admits any kind of ring structure.
\end{remark}

\begin{remark}
    The $\F_2$- and $MU$-homology of the fixed points of $\bpgm{m}$ have been studied by the authors along with Ravenel in \cite{CHR} and \cite{MUhomology}. These computations are quite complicated at heights $>2$, but can be approached via the slice filtration.
\end{remark}

\begin{theorem}\label{thm:introthmfp}
    For all $H\subset G$, the spectrum $\bpgm{m}^H$ is an fp spectrum of type $m|G|/2$. In particular, $eo_h(H)$ is an \emph{fp} spectrum of type $h$.
\end{theorem}

In general, work of Beaudry, the second author, Shi, and Zeng \cite{BHSZ} establishes a $G$-equivariant map $\bpgm{m}\to E_h$, which is a $K(h)$-local equivalence, giving the equivalence of (\ref{eq:K(h)localization}) upon taking fixed points. This allows one to use the much more tractable slice spectral sequence of $\bpgm{m}$ as a connective model of the homotopy fixed point spectral sequence of $E_h$ with respect to its $G$-action. This result is established by giving formulas for the image of the $v_i$ generators in $BP_*$ in the underlying homotopy of $BP^{(\!(G)\!)}$. Section \ref{sec:algebraicstuff} is devoted to using these same formulas to prove the following.

\begin{theorem}\label{thm:introthmbpgmhomotopy}
    The quotient $\pi_*^e\bpgm{m}/(2,v_1,\ldots,v_h)$ is a finite dimensional $\F_2$-vector space of odd dimension.
\end{theorem}

Theorem \ref{thm:introthmbpgmhomotopy} is the algebraic underpinning of Theorem \ref{thm:introthmfp}. In fact, we identify the dimension explicitly in terms of Gaussian binomial coefficients (see Corollary \ref{cor:bpgmodd}).

\subsubsection{Koszul resolutions} In the case $G=C_2$, Theorem \ref{thm:introthmfp} may be obtained inductively by use of cofiber sequences
\begin{equation}\label{eq:vncofiber}
\Sigma^{(2^m-1)\rho_2}\bprn{m}\xrightarrow{\overline{v}_m}\bprn{m}\to\bprn{m-1},
\end{equation}
where $\bar{v}_m$ is a $C_2$-equivariant lift of $v_m$. In the base case $m=0$, $\bprn{0}^{C_2}=H\Z$ is an \emph{fp} spectrum. Using Theorem \ref{thm:introthmbpgmhomotopy}, one may show that $\overline{v}_m$ acts nilpotently on $\bprn{m}\otimes F$, where $F$ is a finite spectrum of chromatic type $>m$. It follows then from the assumption that $\bprn{m-1}^{C_2}$ is an \emph{fp} spectrum and the fact that \emph{fp} spectra form a thick subcategory that $\bprn{m}^{C_2}$ is an \emph{fp} spectrum.

When $|G|>2$, the quotient $\bpgm{m-1}$ is not obtained by taking the cofiber of a self map of $\bpgm{m}$. It is a more subtle equivariant construction that cones off all restrictions, transfers, conjugates, and norms of a class in the $C_2$-equivariant homotopy of $\bpgm{m}$. While the cofiber sequence of (\ref{eq:vncofiber}) mirrors the canonical resolution 
\[\Sigma^{|x|}k[x]\xrightarrow{x}k[x]\to k\]
of $k$ as a $k[x]$-module, we introduce a filtration on the fiber of the map 
\[\bpgm{m}\to \bpgm{m-1}\]
that mirrors the Koszul resolution of $k$ as a $k[G\cdot x]$-module.

In fact, one way to record such resolutions while keeping track of the $G$-action is as follows: one can view the above resolution as a $k[x]$-linear filtration
\[X^\bullet=(\cdots \to 0\to\Sigma^{|x|}k[x]\xrightarrow{x}k[x])\]
of $k[x]$ with $\mathrm{gr} (X^\bullet)=k\oplus \Sigma^{|x|}k[x]$. For a given $C_2$-action on $k[x]$, one may apply the norm $N:=N_{C_2}^G$, using the $G$-symmetric monoidal structure on filtered $k$-modules, giving a $k[G\cdot x]$-linear filtration
\[NX^\bullet=(\cdots \to 0\to\Sigma^{|Nx|}k[G\cdot x]\to\cdots \to k[G\cdot x])\]
of $k[G\cdot x]$. Since $\mathrm{gr}$ is $G$-symmetric monoidal, one has
\[\mathrm{gr}(NX^\bullet)=N(k\oplus \Sigma^{|x|}k[x]).\]
The functor $N$ does not commute with sums, but this right hand side may be computed explicitly in terms of equivariant summands.

In Section \ref{sec:filtration}, we explicitly identify the associated graded of these normed filtrations, which allows us to interpolate between the equivariant quotient 
\[\bpgm{m}/(G\cdot \bar{t}_m^G)\]
and the ordinary cofiber $\bpgm{m}/(N_{C_2}^G(\bar{t}_m^G))$, where $\bar{t}_m^G$ is a $C_2$-equivariant lift of $t_m^G$ (see Definition \ref{def:bpgm}). In Section \ref{sec:fpsection}, this leads to an inductive proof of Theorem \ref{thm:introthmfp} along the lines of the $BP_\R\ip{m}$ case above.

\subsection{The Euler characteristic for \emph{fp} spectra}  From a conceptual point of view, \emph{fp} spectra participate in a sort of duality with finite spectra. The category of \emph{fp} spectra is a thick subcategory of spectra, and it admits a filtration by the thick subcategories $\fp{h}$ of \emph{fp} spectra of type $\le h$. The tensor product of spectra restricts to a pairing of stable $\infty$-categories:
\begin{equation}\label{eq:pairing}
    \fp{h}\otimes \Sp^{\omega}_{>h}\to\mathrm{Thick}(H\F_p).
\end{equation}

\begin{definition}[Levy]
    The Euler characteristic for \emph{fp} spectra is the pairing
    \[\chi:K_0(\fp{h})\otimes K_0(\Sp^{\omega}_{>h})\to K_0(\mathrm{Thick}(H\F_p))\cong\Z\]
    induced by applying $K_0$ to the pairing (\ref{eq:pairing}).
\end{definition}

\begin{lemma}\label{lem:introlemma}
    The pairing $\chi$ is given explicitly by
    \[\chi_E(K):=\chi([E]\otimes[K])=\sum\limits_i(-1)^i\log_p\big|\pi_i(E\otimes K)\big|\]
for $E\in\fp{h}$ and $K\in \Sp^\omega_{>h}$.
\end{lemma}

This pairing may be used to constrain the existence of finite spectra, as the following example of Levy \cite[Section 7]{levy} shows.

\begin{example}\label{ex:komoore}
    The existence of the \emph{fp} spectrum $ko$ implies that the generalized Moore spectrum $\mathbb{S}/(2,v_1)$ cannot exist. Indeed, one sees from the lemma that $\chi_{ku}(\mathbb{S}/(2,v_1))=1$, but the Wood cofiber sequence implies that \[[ku]=2[ko]\in K_0(\fp{1}),\]
    so that $\chi_{ku}$ is divisible by $2$, a contradiction.
\end{example}

The fixed points of $\bpgm{m}$ are higher height generalizations of $ko$, and in Section \ref{sec:moorespectra}, we use the Koszul filtrations of Section \ref{sec:filtration} to give the following vast generalization of this example.

\begin{theorem}\label{thm:introthmktheoryclass}
   For all $H\subset G$ and $m\ge0$, there is a $K$-theory relation
   \[[\bpgm{m}^e]\equiv |H|\cdot[\bpgm{m}^H]\in K_0(\fp{h})\]
   modulo torsion.
\end{theorem}

In fact, Theorem \ref{thm:introthmbpgmhomotopy} implies that the Euler characteristics $\chi_{BP\ip{h}}$ and $\chi_{\bpgm{m}^e}$ differ by multiplication by an odd number (see Proposition \ref{prop:keyEulercharacteristicrelation}), and Theorems \ref{thm:introthmchi} and \ref{thm:introthmeoh} now follow.

\begin{remark}
    If one can produce a finite spectrum $K$ of chromatic type $h+1$ with the property that $\chi_{\bpgm{m}^G}(K)$ is an odd integer, this would fully verify the conjecture of Levy at all heights $h\not\equiv 2$ mod $4$.
\end{remark}

\begin{remark}\label{rmk:sharpness}
    If Levy's conjecture holds, then the constraints of Theorem \ref{thm:introthmeoh} are the sharpest that can be produced via the method of Example \ref{ex:komoore} at heights $h\not\equiv 2$ mod $4$. However, these constraints are not sharp. For example, Example \ref{ex:komoore} rules out $\mathbb{S}/(2,v_1)$ but not $\mathbb{S}/(2,v_1^2)$, and one knows from the minimal $v_1^4$ periodicity of $ko/2$ that $\mathbb{S}/(2,v_1^2)$ does not exist. If one can similarly determine the minimal periodicities of the $eo_h(H)$ theories, the constraints of Theorem \ref{thm:introthmeoh} could likely be sharpened. 
\end{remark}

\begin{remark}\label{remark:smallcomplexes}
    The work of Burklund--Levy \cite{burklundlevy} shows that $K_0(\Sp^\omega_{>h})[\frac{1}{2}]\cong\Z_{(2)}$ is generated by a class that behaves like the Smith--Toda complex $V_h:=\mathbb{S}/(2,v_1,\ldots,v_h)$ if it existed, in the sense that if $V_h$ were to exist, one would have $\chi_{BP\ip{h}}(V_h)=1$. Our Theorem \ref{thm:introthmchi} restricts the existence of small type $h+1$ complexes that are closer in $K_0(\Sp^\omega_{>h})$ to $V_h$ than the smallest existing generalized Moore spectrum of type $h+1$. 
    
    For example, the obstruction to $\mathbb{S}/2$ admitting a $v_1^1$-self map is the Hopf map $\eta$, but coning off $\eta$, one may form $Y=\mathbb{S}/(2,\eta)$ and the type 2 complex $Y/v_1$. Since $\eta$ has odd degree, this satisfies
    \[[Y/v_1]=2[V_1]\in K_0(\Sp^\omega_{>1})[\frac{1}{2}].\]
    On the other hand, the Adams $v_1^4$-self map of $\mathbb{S}/2$ is minimal, and $[\mathbb{S}/(2,v_1^4)]=4[V_1]$, so $Y/v_1$ is closer to $V_1$ than the smallest generalized Moore spectrum of type $2$. In fact, since $\chi_{BP\ip{1}}$ is divisible by $2$, it follows that $Y$ is as close to $V_1$ as a type $2$ complex can be. Similar remarks hold for the type $3$ complex $Z/v_2$ of Bhattacharya--Egger \cite{Z}. Our result thus constrains the existence of higher height analogues of $Y/v_1$ and $Z/v_2$, and in some sense gives a lower bound on the number of obstructions present in building a Smith--Toda complex at height $h$.
\end{remark}

\subsection{Borel completions of $\bpgm{m}$} In Section \ref{sec:telescope}, we explore consequences of Theorem \ref{thm:introthmfp} for the $\bpgm{m}$'s. For example, Hahn--Wilson showed that if $X$ is any \emph{fp} spectrum of type $h$, the canonical map $X\to L_h^fX$ induces an isomorphison on $\pi_i$ for $i$ sufficiently large \cite[Theorem 3.1.3]{hahnwilson}. 

When $X$ is the fixed points of an $MU^{(\!(G)\!)}$-module, the chromatic localizations are closely related to the homotopy fixed points. In fact, in Section \ref{sec:telescope}, we prove the following.

\begin{theorem}\label{thm:introthmlocalizations}
    Let $M\in\mathrm{Mod}(MU^{(\!(G)\!)})$. For all $n\ge0$ and $H\subset G$ and for all $L\in\{L_{K(n)},L_{T(n)},L_n^f,L_n\}$, the map
    \[L(M^H)\to L(M^{hH})\]
    is an equivalence of spectra. Moreover the maps $L_n^f(M^H)\to L_n(M^H)$ are equivalences of spectra.
\end{theorem}

Mahowald--Rezk conjectured that all \emph{fp} spectra satisfy the conditions of the telescope conjecture, and this verifies their conjecture for the examples considered in this article.

\begin{corollary}
    For all $n\ge0$, $m\ge0$, and $H\subset G$, the map
    \[L_n^f\bpgm{m}^H\to L_n\bpgm{m}^H\]
    is an equivalence.
\end{corollary}

Combined with the Hahn--Wilson result, Theorem \ref{thm:introthmlocalizations} yields the following strong completion statement for the $\bpgm{m}$'s.

\begin{theorem}\label{introthmborel}
    For all $i$ sufficiently large, the map
    \[\underline{\pi}_i\bpgm{m}\to\underline{\pi}_iF(EG_+,\bpgm{m})\]
    is an isomorphism of Mackey functors.
\end{theorem}

In other words, the completion maps $\bpgm{m}^H\to \bpgm{m}^{hH}$ have bounded above fiber, for all subgroups $H\subset G$. This result complements existing results in the literature. The first author showed that when $m=\infty$, these completion maps are equivalences, i.e. $BP^{(\!(G)\!)}$ is Borel-complete \cite{carrickcofree}. Greenlees--Meier showed that when $G=C_2$, the map is an isomorphism for all $i\ge0$, so that $\bprn{m}$ is the connective cover of its Borel completion $F(E{C_2}_+,\bprn{m})$ \cite[Proposition 4.9]{GM}. We do not know if the analogue of the Greenlees--Meier result holds for the $\bpgm{m}$'s, but one has the following immediate corollary of Theorem \ref{introthmborel}.

\begin{corollary}
    For all $H\subset G$, the spectrum $\tau_{\ge0}(
    \bpgm{m}^{hH})$ is an fp spectrum of type $m|G|/2$.
\end{corollary}

\subsection{Questions}
\begin{enumerate}

\item Are the \emph{fp} spectra $\bpgm{m}^H$ in the thick subcategory generated by $BP\ip{h}$? This would verify that the $\bpgm{m}^H$'s are consistent with the Hahn--Wilson conjecture discussed in \cite{piotrlee}. It seems plausible that the results of Section \ref{sec:algebraicstuff} can be used to show that $\bpgm{m}^e$ is a finite sum of shifts of forms of $BP\ip{h}$, which reduces the question to showing that $\bpgm{m}^H$ is in the thick subcategory generated by $\bpgm{m}^e$.

\item When $h\equiv 2$ mod $4$, the maximal finite $2$-subgroups of $\mathbb{S}_h$ are isomorphic to the quaternion group $Q_8$. This case is exceptional in that these are the only non-cyclic subgroups of $\mathbb{S}_h$ for any height $h$ and prime $p$. Are there suitable $Q_8$-equivariant quotients of $BP^{(\!(Q_8)\!)}$ that can be used to constrain the image of (\ref{eq:eulerbph}) at these heights?
\item There should be a suitable height filtration of a moduli stack of formal group laws equipped with group actions that encodes the results of Beaudry, the second author, Shi and Zeng in \cite{BHSZ}. The Gaussian binomial coefficients appearing in Section \ref{sec:algebraicstuff} should then give the degrees of the covers obtained by forgetting the group actions down to the classical height filtration. Is there a modular interpretation of these coefficients in this setting? See Example \ref{ex:formalgroupbpc4}.

\item Can one sharpen the constraints of Theorem \ref{thm:introthmmoore} at height $4$ using the periodicity of $\bpcm{2}$ determined by the second author, Shi, Wang, and Xu in \cite{hswx}? See Remark \ref{rmk:sharpness}.
\end{enumerate}

\section{The underlying homotopy of $\bpgm{m}$}\label{sec:algebraicstuff}
In this section, we prove a number of purely algebraic results about the underlying homotopy groups of $BP^{(\!(G)\!)}$ and $\bpgm{m}$. Since the results of this section concern only the underlying homotopy groups, we need only remind the reader that
\[\pi_*^e\bpgm{m}=\Z[G\cdot t_1^G,\ldots,G\cdot t_m^G],\]
where $|t_i^G|=2(2^i-1)$. 

\subsection{Regularity and finiteness} The underlying spectra of $BP^{(\!(G)\!)}$ and $\bpgm{m}$ are complex-oriented ring spectra, and work of Beaudry, the second author, Shi, and Wang gives recursive formulas for the image of the $v_i$ generators in $\pi_*^eBP^{(\!(G)\!)}$ \cite[Theorem 1.1]{BHSZ}. More generally, for any $H\subset G$, there are canonical maps of $H$-spectra
\[BP^{(\!(H)\!)}\to \mathrm{Res}^G_HBP^{(\!(G)\!)}\]
by properties of the norm functor, and they give recursive formulas for the effect of these maps on underlying homotopy groups. We use these to show the following.

\begin{theorem}\label{thm:mainalgebraicthm}
In the ring $\pi_*^eBP^{(\!(G)\!)}=\Z[G\cdot t_i^G| i\ge 1]$, the elements $t_1^G,\ldots,t_m^G$ are nilpotent mod the ideal 
\[I:=(2,v_1,\ldots,v_{h},G\cdot t_{m+1}^G,\ldots,G\cdot t_{h}^G).\]
\end{theorem}
\begin{proof}
We proceed by induction on $|G|$ and $m$. When $G=C_2$, we have 
\[t_i^G\equiv v_i\mod (2,v_1,\ldots,v_{i-1})\]
by \cite[Proposition 3.5]{BHSZ}, and the claim is trivial for all $m$. For any $G$, when $m=0$, there is nothing to prove. Now taking $m>0$ and $|G|>2$, we have by induction that the elements $t_1^{G'},\ldots,t_{2m}^{G'}$ are nilpotent in the ring $\Z[G'\cdot t_i^{G'}| i\ge 1]$ mod the ideal
\[(2,v_1,\ldots,v_{h},G'\cdot t_{2m+1}^{G'},\ldots,G'\cdot t_{h}^{G'})\]
The same is true in the ring $\Z[G\cdot t_i^{G}| i\ge 1]$ via the canonical $BP_*=\pi_*^eBP^{(\!(C_2)\!)}$-linear inclusion 
\[\pi_*^eBP^{(\!(G')\!)}=\Z[G'\cdot t_i^{G'}]\to \Z[G\cdot t_i^G]=\pi_*^eBP^{(\!(G)\!)}\]
Using the $k$-th relation in \cite[Theorem 1.1]{BHSZ} for the values $2m+1\le k\le h$, we see that the elements $t_{2m+1}^{G'},\ldots,t_{h}^{G'}$ are zero mod $I$. Hence by raising to powers, we may assume without loss of generality that the elements $t_1^{G'},\ldots,t_{2m}^{G'}$ are zero in the ring $\Z[G\cdot t_i^{G}]/I$.

By the inductive hypothesis on $m$, we have that the elements $t_1^{G},\ldots,t_{m-1}^G$ are nilpotent mod the ideal
\[(2,v_1,\ldots,v_{(m-1)|G|/2},G\cdot t_{m}^G,\ldots,G\cdot t_{(m-1)|G|/2}^G)\]
and therefore also mod the ideal
\[(2,v_1,\ldots,v_{h},G\cdot t_{m}^G,\ldots,G\cdot t_{h}^G)=I+(G\cdot t_m^G)\]
It will therefore suffice to show that $G\cdot t_m^G$ is nilpotent mod $I$, and since the ideal $(2,v_1,\ldots,v_{h})$ is invariant under the action of $G$ by \cite[Proposition 3.7]{BHSZ}, it suffices to show that $t_m^{G}$ is nilpotent mod $I$.

This is equivalent to showing that the localization $\Z[G\cdot t_i^G][(t_m^G)^{-1}]/I$ is the zero ring. We first show by downward induction on $i$ that $\gamma t_i^{G}=0\in \Z[G\cdot t_i^G][(t_m^G)^{-1}]/I$ for all $1\le i\le m$. For the base case $i=m$, the $2m$-th relation in \cite[Theorem 1.1]{BHSZ} gives us the equation
\[\gamma t_m^{G}\cdot (t_m^{G})^{2^m}=0\in \Z[G\cdot t_i^G]/I\]
which implies that $\gamma t_m^{G}=0\in \Z[G\cdot t_i^G][(t_m^G)^{-1}]/I$. For $i<m$, the $(m+i)$-th relation in \cite[Theorem 1.1]{BHSZ} gives us the equation
\[\gamma t_i^{G}\cdot (t_m^{G})^{2^i}+\gamma t_{i+1}^{G}\cdot (t_{m-1}^{G})^{2^{i+1}}+\cdots+\gamma t_m^{G}\cdot (t_i^{G})^{2^m}=0\in \Z[G\cdot t_i^G]/I\] Assuming by induction that $\gamma t_j^{G}=0\in \Z[G\cdot t_i^G][(t_m^G)^{-1}]/I$ for all $j>i$, we therefore have
\[\gamma t_i^{G}\cdot (t_m^{G})^{2^i}=0\in \Z[G\cdot t_i^G][(t_m^G)^{-1}]/I\]
which implies $\gamma t_i^{G}=0\in \Z[G\cdot t_i^G][(t_m^G)^{-1}]/I$.

Using this, a simple induction argument using the $i$-th relation in \cite[Theorem 1.1]{BHSZ} for $1\le i\le m$ now implies that $t_i^{G}=0\in \Z[G\cdot t_i^G][(t_m^G)^{-1}]/I$ for all $1\le i\le m$. In particular, $t_m^{G}$ is both a unit and zero in $\Z[G\cdot t_i^G][(t_m^G)^{-1}]/I$, so $\Z[G\cdot t_i^G][(t_m^G)^{-1}]/I$ is the zero ring.
\end{proof}

\begin{corollary}\label{cor:bpgmfinite}
    The ring \(\pi_*^e\bpgm{m}/(2,v_1,\ldots,v_{h})\) is finite.
\end{corollary}
\begin{proof}
    This commutative ring is generated as an $\F_2$-algebra by the finite list of elements $G\cdot t_1^G,\ldots,G\cdot t_m^G$, all of which are nilpotent by the previous theorem.
\end{proof}

\begin{corollary}\label{cor:bpgmregularsequence}
    The sequence $(2,v_1,\ldots,v_{h})$ is regular in $\pi_*^e\bpgm{m}$.
\end{corollary}
\begin{proof}
    The element $2$ is a nonzero divisor in $\pi_*^e\bpgm{m}$, so we will show that $(v_1,\ldots,v_{h})$ is a regular sequence in
\[\pi_*^e\bpgm{m}/2=\F_2[G\cdot t_1^G,\ldots,G\cdot t_m^G]\]
This follows from the fact that the latter is a polynomial ring of Krull dimension $h$, and the quotient by the ideal $(v_1,\ldots,v_{h})$ gives a Krull dimension zero ring by Corollary \ref{cor:bpgmfinite}. Indeed, for each $i$, the ring $\F_2[G\cdot t_1^G,\ldots,G\cdot t_m^G]/(v_1,\ldots,v_i)$ is a natural subring of $\F_2[[G\cdot t_1^G,\ldots,G\cdot t_m^G]]/(v_1,\ldots,v_i)$, so the regularity claim follows from the corresponding one in $\F_2[[G\cdot t_1^G,\ldots,G\cdot t_m^G]]$, which is a complete Noetherian regular local ring of Krull dimension $h$, with each $v_i$ contained in the maximal ideal $\mathfrak{m}=(G\cdot t_1^G,\ldots,G\cdot t_m^G)$.
\end{proof}

\subsection{Poincaré series and dimension}
Our Corollary \ref{cor:bpgmfinite} implies that 
\[\pi_*^e\bpgm{m}/(2,v_1,\ldots,v_{h})\] is a finite dimensional $\F_2$-vector space. In fact, by using Poincaré series, we can compute its dimension explicitly in terms of Gaussian binomial coefficients.

\begin{definition}
   For a $\Z_{\ge0}$-graded $\F_2$-vector space $V_*$, let
    \[f_{V_*}(x)=\sum\limits_{n\ge0}(\dim_{\F_2}V_n) x^n\in\F_2[[x]]\]
    be its \emph{Poincaré series}. 
\end{definition}

For the Poincaré series we consider below, we work with graded vector spaces concentrated in even degrees, so we implicitly divide all degrees by 2. Note that this does not change the dimension of the underlying vector space.

\begin{proposition}\label{prop:poincarebpgm}
    The Poincaré series $f_m(x)$ of $\pi_*^e\bpgm{m}/(2,v_1,\ldots,v_{h})$ is given by
\[f_m(x)=\frac{\prod\limits_{i=1}^{h}1-x^{2^i-1}}{\bigg(\prod\limits_{i=1}^m1-x^{2^i-1}\bigg)^{|G|/2}}\]
\end{proposition}    
\begin{proof}
    By Corollary \ref{cor:bpgmregularsequence}, there is an isomorphism of graded $\F_2$-vector spaces
\[\F_2[G\cdot t_1^G,\ldots,G\cdot t_m^G]\cong \pi_*^e\bpgm{m}/(2,v_1,\ldots,v_{h})\otimes_{\F_2}\F_2[v_1,\ldots,v_{h}]\]
giving an equation
\[\bigg(\prod\limits_{i=1}^{m}\frac{1}{1-x^{2^i-1}}\bigg)^{|G|/2}=f_m(x)\cdot \prod\limits_{i=1}^{h}\frac{1}{1-x^{2^i-1}}\in \F_2[[x]],\]
and the claim is obtained by dividing in $\F_2[[x]]$.
\end{proof}

\begin{definition}
    Let 
    \[\binom{n}{m}_2=\frac{(1-2^n)(1-2^{n-1})\cdots(1-2^{n-m+1})}{(1-2)(1-2^2)\cdots(1-2^m)}\]
    be the Gaussian $q$-binomial coefficient with $q=2$.
\end{definition}

\begin{remark}
    For all $n\ge m$, $\binom{n}{m}_2$ is a positive integer, which is therefore odd as it divides the odd number $(1-2^n)(1-2^{n-1})\cdots(1-2^{n-m+1})$.
\end{remark}

\begin{corollary}\label{cor:bpgmodd}
    For all $m\ge0$, one has
    \[\dim_{\F_2}\big(\pi_*^e\bpgm{m}/(2,v_1,\ldots,v_{h})\big)=\prod\limits_{j=0}^{|G|/2-1}\binom{jm}{m}_2\]
    In particular, this $\F_2$-vector space has odd dimension.
\end{corollary}
\begin{proof}
    We prove the claim via l'Hôpital's rule. We have
    \begin{align*}
        \lim_{x\to 1}f_m(x)&=\lim_{x\to 1}\frac{\prod\limits_{i=1}^{h}1-x^{2^i-1}}{\bigg(\prod\limits_{i=1}^m1-x^{2^i-1}\bigg)^{|G|/2}}\\
        &=\prod\limits_{j=0}^{|G|/2-1}\lim_{x\to 1}\prod\limits_{i=1}^m\frac{1-x^{2^{jm+i}-1}}{1-x^{2^i-1}}\\
        &=\prod\limits_{j=0}^{|G|/2-1}\prod\limits_{i=1}^m\lim_{x\to 1}\frac{1-x^{2^{jm+i}-1}}{1-x^{2^i-1}}\\
        &=\prod\limits_{j=0}^{|G|/2-1}\prod\limits_{i=1}^m\frac{1-2^{jm+i}}{1-2^i}\\
        &=\prod\limits_{j=0}^{|G|/2-1}\binom{jm}{m}_2
    \end{align*}
\end{proof}

\begin{example}\label{ex:formalgroupbpc4}
    In the case $G=C_4$, the ring $\pi_*^e\bpcm{m}/2$ is isomorphic to the quotient
    \[\F_2[\xi_1,\ldots,\xi_m]/(\zeta_{m+1},\ldots,\zeta_{2m})\]
    of the dual Steenrod algebra $\mathcal{A}_*$, where $\xi_i$ are the usual Milnor generators, and $\zeta_i$ are the Hopf conjugates. The group scheme $\Spec(\mathcal{A}_*)$ is the (strict) automorphism group of the additive formal group $\widehat{\mathbb{G}_a}$ over $\F_2$. Under this identification, $\Spec(\pi_*^e\bpcm{m}/2)$ is the subscheme consisting of those automorphisms of the form
    \[f(x)=x+\xi_1x^2+\cdots +\xi_mx^{2^m}\]
    with the property that $f^{-1}(x)$ has the same form 
    \[f(x)=x+\zeta_1x^2+\cdots +\zeta_mx^{2^m}.\]
    Corollary \ref{cor:bpgmodd} computes the dimension over $\F_2$ of the ring of functions on this scheme, and we speculate that the appearance of Gaussian binomial coefficients has a modular interpretation in these terms.
\end{example}

\section{A filtration for equivariant quotients}\label{sec:filtration}
We recall the twisted monoid construction $M/(G\cdot \bar{x})$ of the second author, Hopkins, and Ravenel, and we introduce the Koszul filtration on the fiber of the quotient map
\[
    M\to M/(G\cdot \bar{x}).
\]

\subsection{Equivariant quotients} 

\begin{definition}\label{def:equivariantcofiber}
    Let $\mathbb{S}[\bar{x}]$ be the free $\mathbb{E}_1$-algebra in $\Sp^{C_2}$ on a class $\bar{x}$ in degree $k\rho_2$. We use the notation
    \[\mathbb{S}[G\cdot\bar{x}]:=\begin{cases}N_{C_2}^G(\mathbb{S}[\bar{x}])& C_2\subset G\\\mathrm{Res}^{C_2}_e(\mathbb{S}[\bar{x}])&G=\{e\}\end{cases}\]
    Note that $\mathbb{S}[\{e\}\cdot\bar{x}]=\mathbb{S}[x]$ is the free $\mathbb{E}_1$-algebra in $\Sp$ on the underlying class $x:=\mathrm{Res}^{C_2}_e(\bar{x})$. Moreover, since $N_{C_2}^G$ is monoidal, $\mathbb{S}[G\cdot\bar{x}]$ is canonically an augmented $\mathbb{E}_1$-algebra in $\Sp^G$.
    
    For classes $\bar{x}_i$ in degrees $k_i\rho_2$, we define the augmented $\mathbb{E}_1$-algebra
    \[\mathbb{S}[G\cdot\bar{x}_1,\ldots,G\cdot\bar{x}_n]:=\bigotimes\limits_{i=1}^n\mathbb{S}[G\cdot\bar{x}_i],\]
    and for $M\in\mathrm{Mod}(\mathbb{S}[G\cdot\bar{x}_1,\ldots,G\cdot\bar{x}_n])$, we define
    \[M/(G\cdot\bar{x}_1,\ldots,G\cdot\bar{x}_n):=M\otimes_{\mathbb{S}[G\cdot\bar{x}_1,\ldots,G\cdot\bar{x}_n]}\mathbb{S}.\]
\end{definition}

\begin{example}\label{example:equivariantcofibers}
    When $G\subset C_2$, the quotient $M/(G\cdot\bar{x}_1,\ldots,G\cdot\bar{x}_n)$ is the ordinary (iterated) cofiber by the elements $\bar{x}_i$ when $G=C_2$, and $x_i:=\mathrm{Res}^{C_2}_e(\bar{x}_i)$ when $G=\{e\}$, respectively. More generally, it follows from the definitions that
    \[\mathrm{Res}^G_H(M/(G\cdot\bar{x}_1,\ldots,G\cdot\bar{x}_n))=M/(H\cdot g\bar{x}_1,\ldots,H\cdot g\bar{x}_n\mid gH\in G/H).\]
    For example, when $G=C_4$, the underlying $C_2$-spectrum of $M/(C_4\cdot\bar{x})$ is the ordinary (iterated) cofiber $\mathrm{Res}^{C_4}_{C_2}(M)/(\bar{x},\gamma\bar{x})$, where $\gamma$ is a generator of $C_4$.
\end{example}

It is shown in \cite[Section 5]{HHR} that there exist classes $\bar{t}_i^G\in\pi_{*\rho_2}^{C_2}BP^{(\!(G)\!)}$ such that
\[\pi_{*\rho_2}^{C_2}BP^{(\!(G)\!)}=\Z[G\cdot\bar{t}_1^G,G\cdot\bar{t}_2^G,\ldots],\]
and $\mathrm{Res}(\bar{t}_i^G)=t_i^G$, where the $t_i^G$'s are the underlying generators discussed in Section \ref{sec:algebraicstuff}.

\begin{definition}\label{def:bpgm}
    We define the $MU^{(\!(G)\!)}$-module
    \[\bpgm{m}:=\mathrm{colim}_kBP^{(\!(G)\!)}/(G\cdot\bar{t}_m^G,\ldots,G\cdot\bar{t}_{m+k}^G)\]
\end{definition}

\begin{remark}
    Note in particular that
    \(\bpgm{m-1}=\bpgm{m}/(G\cdot\bar{t}_m^G)\).
\end{remark}

\subsection{Norms of quotients via filtrations} 
The parametrized Day convolution developed by Nardin and Shah \cite[Section 3]{nardinshah} immediately gives the following facts about equivariant filtrations.

\begin{proposition}
    The categories $\mathrm{Fil}(\Sp^G)$ and $\mathrm{gr}(\Sp^G)$ of filtered and graded $G$-spectra respectively admit the structure of $G$-symmetric monoidal $\infty$-categories with the property that the functors
\begin{align*}
    \mathrm{colim}:\mathrm{Fil}(\Sp^G)&\to\Sp^G\\
    X^\bullet&\mapsto X^{-\infty}\\
    \mathrm{gr}:\mathrm{Fil}(\Sp^G)&\to\mathrm{gr}(\Sp^G)\\
    X^\bullet&\mapsto (\mathrm{cofib}(X^{i+1}\to X^i))_{i\in\Z}\\
    \mathrm{forget}:\mathrm{gr}(\Sp^G)&\to\Sp^G\\
    (X^i)_{i\in\Z}&\mapsto\bigoplus\limits_{i\in\Z}X^i
\end{align*}
have canonical $G$-symmetric monoidal structures.
\end{proposition}

\begin{remark}
    The norm functors in $\mathrm{gr}(\Sp^G)$ are computed in terms of the norms in $\Sp^G$ using the fact that $\mathrm{forget}:\mathrm{gr}(\Sp^G)\to\Sp^G$ is $G$-symmetric monoidal and by use of the distributive laws of the second author, Hopkins, and Ravenel, which describe the behavior of the norm functors in $\Sp^G$ with respect to sums (see \cite[Proposition A.37]{HHR}).
\end{remark}

We begin by unpacking the distributive law of \cite[Proposition A.37]{HHR} in our case of interest, which yields a decomposition:

\begin{equation}\label{eq:distributive}
    N_{C_2}^{G}\big(\mathbb S\oplus E\big)\simeq \bigoplus_{f\in\Map^{C_2}\big(G,\{0,1\}\big)} E^{\otimes f^{-1}(1)},
\end{equation}
where $\{0,1\}$ has a trivial $C_2$-action, and $\Map^{C_2}(-,-)$ denotes $C_2$-equivariant maps. When we view this as a graded object, with \(\mathbb S\) in grading \(0\) and \(E\) in grading \(1\), then the grading of a summand associated to \(f\) is $\frac{1}{2}\sum_{g\in G} f(g)$.

In fact, the sum and the tensor appearing in (\ref{eq:distributive}) are  their indexed variants. For the sum, each of the $G$-conjugates of a $C_2$-equivariant function \(f\colon G\to \{0,1\}\) is grouped via induction in a single equivariant summand
\[\mathrm{Ind}_{H_f}^GE^{\otimes f^{-1}(1)},\]
where $H_f$ is the stabilizer of $f$ in $G$. For the tensor, factors are grouped via the norm.

\begin{definition}
    The indexed tensor summand associated to \(f\) is the \(H_f\)-equivariant spectrum
    \[
        E^{\otimes f^{-1}(1)}:=\bigotimes_{gH_f\in f^{-1}(1)/H_f} g\cdot N_{C_2}^{H_f}(E).
    \]
\end{definition}
  
When we apply this to \(E=\Sigma^{k\rho_2}\mathbb S^0[\bar{x}]\) for a \(C_2\)-equivariant class \(\bar{x}\) in degree $k\rho_2$, then this can be rewritten in the notation of Definition \ref{def:equivariantcofiber} as
\[
    \big(\Sigma^{k\rho_2}\mathbb S^0[\bar{x}]\big)^{\otimes f^{-1}(1)}=\Sigma^{n_f\cdot k\rho_{H_f}}\mathbb S^0[H_f\cdot g\bar{x}\mid gH_f\in f^{-1}(1)/H_f].
\]
where we set $n_f:=|f^{-1}(1)/H_f|$. Putting this all together, this gives us a description of the norm of the associated graded as a graded object.

\begin{proposition}
    The norm $N_{C_2}^G$ of the graded $C_2$-spectrum $\mathbb{S}\oplus\Sigma^{k\rho_2}\mathbb{S}[\overline{x}]$ with $\mathbb{S}$ in grading $0$ and $\Sigma^{k\rho_2}\mathbb{S}[\bar{x}]$ in grading $1$ is
    \[
        \bigoplus_{f\in\Map^{C_2}\big(G,\{0,1\}\big)/G} \Ind_{H_f}^{G} \Sigma^{n_f\cdot k\rho_{H_f}}\mathbb S^0[H_f\cdot g\bar{x}\mid gH_f\in f^{-1}(1)/H_f],
    \]
    with the summand associated to \(f\) in grading \(\tfrac{1}{2}|f^{-1}(1)|\).
\end{proposition}

The $\mathbb{S}[G\cdot \overline{x}]$-module structure on the summand 
\[\Ind_{H_f}^{G} \mathbb S^0[H_f\cdot g\bar{x}\mid gH_f\in f^{-1}(1)/H_f]\]
is induced from  the $\mathrm{Res}^G_{H_f}\mathbb{S}[G\cdot \overline{x}]=\mathbb{S}[H_f\cdot g\overline{x}\mid gH_f\in G/H_f]$-module structure on $S^0[H_f\cdot g\bar{x}\mid gH_f\in f^{-1}(1)/H_f]$ furnished by the ring map
\[\mathbb{S}[H_f\cdot g\overline{x}\mid gH_f\in G/H_f]\to \mathbb{S}[H_f\cdot g\overline{x}\mid gH_f\in f^{-1}(1)/H_f],\]
which quotients out $(H_f\cdot g\overline{x}\mid gH_f\in f^{-1}(0)/H_f)$. Tensoring this filtration with an $\mathbb{S}[G\cdot\bar{x}]$-module $M$, this yields the following.

\begin{proposition}\label{prop:associatedgradedformula}
    Let $M$ be an $\mathbb{S}[G\cdot x]$-module with $|x|=k\rho_2$. Then $M$ admits a $\mathbb{S}[G\cdot x]$-linear filtration with associated graded
   \[
        \bigoplus_{f\in\Map^{C_2}\big(G,\{0,1\}\big)/G} \Ind_{H_f}^{G} \Sigma^{n_f\cdot k\rho_{H_f}}M/(H_f\cdot g\bar{x}\mid gH_f\in f^{-1}(0)/H_f),
    \]
    with the summand associated to \(f\) in grading \(\tfrac{1}{2}|f^{-1}(1)|\).
\end{proposition}

\begin{example}
    Let $M$ be an $\mathbb{S}[C_4\cdot \bar{x}]$-module with $|\bar{x}|=k\rho_2$. Then $M$ admits a $\mathbb{S}[C_{4}\cdot \bar{x}]$-linear filtration with associated graded
    \[
    \mathrm{gr}_m=
    \begin{cases}
        M/(C_{4}\cdot \bar{x})&m=0\\
        \mathrm{Ind}_{C_2}^{C_4}\Sigma^{k\rho_2}M/\bar{x}&m=1\\
        \Sigma^{k\rho_{4}}M&m=2
    \end{cases}
    \]
\end{example}

\begin{example}
    Let $M$ be an $\mathbb{S}[C_8\cdot \bar{x}]$-module with $|\bar{x}|=k\rho_2$. Then $M$ admits a $\mathbb{S}[C_{8}\cdot \bar{x}]$-linear filtration with associated graded
    \[
    \mathrm{gr}_m=
    \begin{cases}
        M/(C_{8}\cdot \bar{x})&m=0\\
        \mathrm{Ind}_{C_2}^{C_8}\Sigma^{k\rho_2}\frac{M}{(\bar{x},\gamma \bar{x},\gamma^2 \bar{x})}&m=1\\
        \mathrm{Ind}_{C_4}^{C_8}\Sigma^{k\rho_4}M/(C_4\cdot \bar{x})\oplus \mathrm{Ind}_{C_2}^{C_8}\Sigma^{2k\rho_2}\frac{M}{(\bar{x},\gamma \bar{x})}&m=2\\
        \mathrm{Ind}_{C_2}^{C_8}\Sigma^{3k\rho_2}M/(\bar{x})&m=3\\
        \Sigma^{k\rho_{8}}M&m=4
    \end{cases}
    \]
\end{example}

\subsection{Nilpotence on equivariant quotients} We will use the filtration of Proposition \ref{prop:associatedgradedformula} to prove a generalization to the equivariant setting of the following well-known fact.

\begin{lemma}\label{lem:x^2actsbyzero}
    Let $R$ be an $\mathbb{E}_1$-ring in a stably monoidal $\infty$-category $(\mathcal{C},\otimes,\mathbf{1})$ and let $x\in[\Sigma^k\mathbf{1},R]_{\mathcal{C}}$. Then $x^2$ acts by zero on $R/x$ in $\mathrm{Mod}_{\mathcal{C}}(R)$. 
\end{lemma}
\begin{proof}
    Consider the diagram
    \[
    \begin{tikzcd}
        R\arrow[r]\arrow[d,"x"]&R/x\arrow[r]\arrow[d,"x"]\arrow[dl,dashed]&\Sigma R\arrow[d,"x"]\\
        R\arrow[r]\arrow[d,"x"]&R/x\arrow[r]\arrow[d,"x"]&\Sigma R\arrow[d,"x"]\arrow[dl,dashed]\\
        R\arrow[r]&R/x\arrow[r]&\Sigma R
    \end{tikzcd}
    \]
    in $\mathrm{Mod}_{\mathcal{C}}(R)$, where the rows are cofiber sequences, and we suppress shifts by $k$ for simplicity. The dashed arrows exist by universal property of the cofiber, and they imply that the composite in the middle column factors through the composite in the middle row, which is zero as it is a cofiber sequence.
\end{proof}

We generalize this to the equivariant setting by replacing cofibers $R/x$ with equivariant quotients $R/(G\cdot \bar{x})$.

\begin{proposition}\label{prop:nilpotenceequivariantcofiber}
    Let $R$ be a $G$-$\mathbb{E}_\infty$-ring with $\bar{x}_1,\ldots,\bar{x}_n\in\pi_{*\rho_2}^{C_2}R$, and let $M$ be an $R$-module. Then $N_{C_2}^H(\bar{x}_i)$ acts nilpotently on $\mathrm{Res}^G_H(M/(G\cdot \bar{x}_1,\ldots,G\cdot \bar{x}_n))$ in $\mathrm{Mod}_{\Sp^H}(\mathrm{Res}^G_HR)$ for all $C_2\subset H\subset G$ and $1\le i\le n$.
\end{proposition}
\begin{proof}
    We proceed by induction on $n$ and $|G|$. For induction on $n$, we use the fact that
\[M/(G\cdot \bar{x}_1,\ldots, G\cdot \bar{x}_n)=M/(G\cdot \bar{x}_1,\ldots,G\cdot \bar{x}_{n-1})\otimes_R R/(G\cdot \bar{x}_n)\]
and that if $N_{C_2}^H(\bar{x}_n)$ acts nilpotently on $M$, then it acts nilpotently on the tensor product
\[M/(G\cdot \bar{x}_1,\ldots,G\cdot \bar{x}_{n-1})=M\otimes_RR/(G\cdot \bar{x}_1,\ldots,G\cdot \bar{x}_{n-1})\]
by functoriality. We may therefore take $n=1$.

When $G=C_2$, this is Lemma \ref{lem:x^2actsbyzero}, using the identification of Example \ref{example:equivariantcofibers}. We may therefore take $|G|>2$ and assume the result for all $C_2\subset H\subsetneq G$ and all $n$. For any such subgroup $H$, we have $\mathrm{Res}^G_H (M/(G\cdot \bar{x}))=M/(H\cdot g\bar{x}\mid gH\in G/H)$. It follows then by the inductive hypothesis that $N_{C_2}^H(\bar{x})$ acts nilpotently on $\mathrm{Res}^G_H (M/(G\cdot \bar{x}))$.

    It remains to show that $N_{C_2}^G(\bar{x})$ acts nilpotently on $M/(G\cdot \bar{x})$. The filtration of Proposition \ref{prop:associatedgradedformula} factors the multiplication map
    \[N_{C_2}^G(\bar{x}):M\to M\]
    into a composite of $2^{n-1}$ maps, such that the cofiber of the $m$-th map is given by $\mathrm{gr}_m$. It follows that $M/(G\cdot \bar{x})$ is in the thick subcategory in $\mathrm{Mod}(R)$ spanned by $M/N_{C_2}^G(\bar{x})$ and the $\mathrm{gr}_m$'s for $0<m<2^n$. The class $N_{C_2}^G(\bar{x})^2$ acts by zero on the ordinary cofiber $M/N_{C_2}^G(\bar{x})$ by Lemma \ref{lem:x^2actsbyzero}, so it suffices to show that $N_{C_2}^G(\bar{x})$ acts nilpotently on each of the $\mathrm{gr}_m$'s for $0<m<2^n$. In this case, up to shifts, $\mathrm{gr}^m$ is of the form $\mathrm{Ind}_H^G\mathrm{Res}^G_H(M)/(H\cdot S)$, for $H$ a proper subgroup, and $S$ a set of variables containing $G/H$ conjugates of $\bar{x}$. This therefore follows by the inductive hypothesis.
\end{proof}

With an eye toward applying our Theorem \ref{thm:mainalgebraicthm}, we will need a similar nilpotence statement replacing $x_i$ by any element in the augmentation ideal.

\begin{corollary}\label{cor:nilpotentfromaugideal}
    Let $M\in\mathrm{Mod}_{\Sp^G}(R)$, and let $\bar{x}$ be any class in the ideal
    \[(G\cdot\bar{x}_1,\ldots,G\cdot\bar{x}_n)\subset\pi_{*\rho}^{C_2}R\] generated by the classes $g\bar{x}_j$ for all $1\le j\le n$ and all $g\in G/C_2$. Then $N_{C_2}^H(\bar{x})$ acts nilpotently on $\mathrm{Res}^G_HM/(G\cdot \bar{x}_1,\ldots,G\cdot \bar{x}_n)$ in $\mathrm{Mod}_{\Sp^H}(\mathrm{Res}^G_HR)$ for all $C_2\subset H\subset G$.
\end{corollary}
\begin{proof}
    The class $\bar{x}$ can be written as a linear combination of monomials in the generators $g \bar{x}_j$. Applying the norm $N_{C_2}^H(-)$ to this sum yields a linear combination of transfers of norms of monomials in the generators $g \bar{x}_j$. This concludes the proof as, by Proposition \ref{prop:nilpotenceequivariantcofiber}, each term in this sum acts nilpotently on $\mathrm{Res}^G_HM/(G\cdot \bar{x}_1,\ldots,G\cdot \bar{x}_n)$. 
\end{proof}

\subsection{Thick subcategories} For a $G$-$\mathbb{E}_\infty$-ring $R$ and an $R$-module $M$, we use the notation
\[\mathrm{Thick}^R(M)\]
to denote the thick subcategory of $\mathrm{Mod}_{\Sp^G}(R)$ generated by $M$.

\begin{proposition}\label{prop:thicksubcatsprop}
Let $R\in Sp^{G}$ be a $G$-$\mathbb{E}_\infty$-ring and $M\in\mathrm{Mod}_{\Sp^G}(R)$. Suppose $\bar{x}_1,\ldots,\bar{x}_n\in \pi_\star^{C_2}R$ have the property that the maps
\[N_{C_2}^H(\bar{x}_i):\Sigma^{|N_{C_2}^H(\bar{x}_i)|}M\to M\]
are nilpotent in $\mathrm{Mod}_{\Sp^H}(\mathrm{Res}^G_HR)$ for all $1\le i\le n$ and $C_2\subset H\subset G$. Then
\[M\in \mathrm{Thick}^R_{H\subset G}\big(G/H_+\otimes M/(G\cdot \bar{x}_1,\ldots,G\cdot \bar{x}_n)\big)\]
\end{proposition}

\begin{proof}
Our proof closely follows that of Proposition \ref{prop:nilpotenceequivariantcofiber}. We reduce as before to the case $n=1$ and set $\bar{x}:=\bar{x}_1$. When $G=C_2$, the result is a standard argument using that thick subcategories are closed under retracts. We may therefore assume that the result is known at $C_2\subset H\subsetneq G$, for all $n$.

Since $N_{C_2}^G(\bar{x})$ acts nilpotently on $M$, it follows that
\[M\in \mathrm{Thick}^R_{H\subset G}(G/H_+\otimes M/N_{C_2}^G(x)).\]
The filtration of Proposition \ref{prop:associatedgradedformula} implies that
\[M/N_{C_2}^G(x)\in\mathrm{Thick}^R(\mathrm{gr}_m)_{0<m<2^n},\]
so it suffices to show that $\mathrm{gr}_m\in \mathrm{Thick}^R_{H\subset G}\big(G/H_+\otimes M/(G\cdot \bar{x})\big)$ for all $0<m<2^n$. In this case, up to shifts, $\mathrm{gr}_m$ is of the form $\mathrm{Ind}_H^G\mathrm{Res}^G_H(M)/(H\cdot S)$, for $C_2\subset H\subsetneq G$ and $S$ a set of variables containing $G/H$-conjugates of $\bar{x}$. The class $N_{C_2}^H(\bar{x})$ acts nilpotently on $\mathrm{Res}^G_HM$ by assumption and therefore the same is true for all of its $G/H$-conjugates. By functoriality, these classes also act nilpotently on the quotient $\mathrm{Res}^G_H(M)/(H\cdot S)$. It follows by the induction hypothesis that 
\[\mathrm{Res}^G_H(M)/(H\cdot S)\in\mathrm{Thick}^{\mathrm{Res}^G_H(R)}_{K\subset H}(H/K_+\otimes \mathrm{Res}^G_H(M/(G\cdot\bar{x}))),\]
since $\mathrm{Res}^G_H(M/(G\cdot\bar{x}))$ is of the form $\big(\mathrm{Res}^G_H(M)/(H\cdot S)\big)/(H\cdot S')$ for $S'$ some set of variables containing $G/H$-conjugates of $\bar{x}$. Since the left adjoint 
\[\mathrm{Ind}_H^G:\mathrm{Mod}_{\Sp^H}(\mathrm{Res}^G_HR)\to\mathrm{Mod}_{\Sp^G}(R)\]
is exact, it follows that
\[\mathrm{Ind}_H^G\mathrm{Res}^G_H(M)/(H\cdot S)\in\mathrm{Thick}^R_{H\subset G}(G/H_+\otimes M/(G\cdot \bar{x})),\]
completing the induction.\qedhere

\end{proof}

\section{The $\bpgm{m}$'s are \emph{fp}-spectra}\label{sec:fpsection}
In this section, we carry out an induction argument to show that $\bpgm{m}^H$ is an \emph{fp} spectrum for all $H\subset G$ and $m\ge0$. When $m=0$, this always gives the Eilenberg--Maclane spectrum $H\Z$, which is an \emph{fp} spectrum, and we can reduce to this case using Proposition \ref{prop:thicksubcatsprop}. This requires us to know that the norms of the classes $\overline{t}_i^G\in \pi_{*\rho_2}^{C_2}BP^{(\!(G)\!)}$ act nilpotently on $\bpgm{m}$ for $1\le i\le m$. Algebraically, this is a consequence of Theorem \ref{thm:mainalgebraicthm}, and we lift this to the spectrum level in this section.

Up to taking powers, modding out by the ideal $(2,v_1,\ldots,v_{h})$ is realized on the spectrum level by tensoring with a finite spectrum $K$ of chromatic type $h+1$. Theorem \ref{thm:mainalgebraicthm} states that, up to raising to powers, the classes $\overline{t}_i^G$ lie in the ideal $(G\cdot \overline{t}_{m+1}^G,\ldots,G\cdot \overline{t}_{h}^G)$ after modding out by $(2,v_1,\ldots,v_{h})$, which takes the following form on the spectrum level. In the following, recall from \cite[Theorem 2.28]{HK} that, given a choice of $v_i$ generators in $BP_*$, there exist unique lifts $\bar{v}_i\in\pi_{(2^i-1)\rho}^{C_2}BP_\R$ along the restriction map.

\begin{proposition}\label{prop:ticlasses}
Let $K$ be a finite, homotopy commutative $C_2$ ring spectrum such that $i^*_eK$ has chromatic type $h+1$. Then, for all $i\le m$, there exists $N_i>>0$ and 
\[x_i\in(G\cdot \overline{t}_{m+1}^G,\ldots,G\cdot \overline{t}_{h}^G)\subset\pi_{*\rho_2}^{C_2}BP^{(\!(G)\!)}\]
such that the unit map
\[\eta:\pi_{*\rho_2}^{C_2}BP^{(\!(G)\!)}\to \pi_{*\rho_2}^{C_2}(BP^{(\!(G)\!)}\otimes K)
\]
satisfies $\eta(\overline{t}_i^{N_i})=\eta(x_i)$.
\end{proposition} 
\begin{proof}
We first note that each of the elements $2,\overline{v}_1,\ldots,\overline{v}_{h}$ is nilpotent in the ring $\pi_{*\rho_2}^{C_2}(BP^{(\!(G)\!)}\otimes K)$, since for each $0\le i\le h$, the $C_2$-spectrum
\[(\mathrm{Res}^G_{C_2}BP^{(\!(G)\!)}\otimes K)[\overline{v}_i^{-1}]\]
is contractible. Indeed, it is a module over 
\[(\BPR\otimes K)[\overline{v}_i^{-1}],\]
which is contractible since $\Phi^{C_2}(\overline{v}_i)=0$ and $\mathrm{Res}^G_e\BPR[\overline{v}_i^{-1}]=BP[v_i^{-1}]$, and $BP[v_i^{-1}]$ is Bousfield equivalent to the Johnson-Wilson theory $E(i)$. It follows that, for all $0\le i\le h$, there exist $k_i>>0$ such that $\eta(\overline{v}_i^{k_i})=0$. Moreover, since $BP^{(\!(G)\!)}\otimes K$ is homotopy commutative, the ring $\pi_{*\rho_2}(BP^{(\!(G)\!)}\otimes K)$ is graded-commutative, so if $x\in (2,\bar{v}_1,\ldots,\bar{v}_{h})$, then $\eta(x)$ is nilpotent.

 By Theorem \ref{thm:mainalgebraicthm}, for all $1\le i\le m$, there exists $n_i>>0$ such that
\[\overline{t}_i^{n_i}=w_i+y_i,\]
for some $w_i\in (2,\overline{v}_1,\ldots,\overline{v}_{h})$ and $y_i\in (G\cdot \overline{t}_{m+1}^G,\ldots,G\cdot \overline{t}_{h}^G)$. 
 It follows then that $\eta(w_i^{M_i})=0$ for some $M_i>>0$, and by replacing $n_i$ with a larger power $N_i$, it follows that $\eta(\overline{t}_i^{N_i})=\eta(x_i)$, for some 
\[x_i\in (G\cdot \overline{t}_{m+1}^G,\ldots,G\cdot \overline{t}_{h}^G),\]
as claimed.
\end{proof}

\begin{example}
If $K$ is a finite (nonequivariant) homotopy commutative ring spectrum, both $\mathrm{inf}_e^{C_2}K$ and $N_e^{C_2}K$ work here. Such finite spectra $K$ of any chromatic type exist by work of Burklund \cite{burklundquotients}.
\end{example}

This puts us now in the situation to apply Proposition \ref{prop:thicksubcatsprop}. In the following, we let $K$ be a finite, homotopy commutative ring spectrum of type $h+1$. 

\begin{lemma}\label{lemma:nilpotentactiononbpgm}
Kor all $C_2\subset H\subset G$, $N_{C_2}^H(\overline{t_i}^G)$ acts nilpotently on \[\mathrm{Res}^G_H(\bpgm{m})\otimes \mathrm{inf}_e^HK\in\mathrm{Mod}_{\Sp^H}(\mathrm{Res}^G_HMU^{(\!(G)\!)})\]
for $1\le i\le m$.
\end{lemma}
\begin{proof}
By Proposition \ref{prop:ticlasses}, the unit map 
\[\eta:\pi_{*\rho}^{C_2}(MU^{(\!(G)\!)}_{(2)})\to \pi_{*\rho}^{C_2}(BP^{(\!(G)\!)}\otimes \mathrm{inf}_e^{C_2}K)\]
has the property that
\[\eta((\overline{t}_i^G)^{N_i})=\eta(x_i)\]
for some exponent $N_i$ and some 
\[x_i\in (G\cdot \overline{t}_{m+1}^G,\ldots,G\cdot \overline{t}_{h}^G)\in\pi_{*\rho}^{C_2}(MU^{(\!(G)\!)}_{(2)}),\]
for all $1\le i\le m$. Since $\mathrm{Res}^G_HBP^{(\!(G)\!)}\ip{m}\otimes N_{C_2}^H(\mathrm{inf}_e^{C_2}K)$ is a module over \[\mathrm{Res}^G_HBP^{(\!(G)\!)}\otimes N_{C_2}^H(\mathrm{inf}_e^{C_2}K)\] 
in $\mathrm{Mod}_{\Sp^H}(\mathrm{Res}^G_HMU^{(\!(G)\!)})$, it follows that the action of $N_{C_2}^H((\overline{t}_i^G)^{N_i})$ on \[\mathrm{Res}^G_H(BP^{(\!(G)\!)}\ip{m})\otimes N_{C_2}^H(\mathrm{inf}_e^{C_2}K)\]
is the same as that of $N_{C_2}^H(x_i)$.

We claim then that $N_{C_2}^H(x_i)$ acts nilpotently on
\[\mathrm{Res}^G_H(BP^{(\!(G)\!)}\ip{m})\otimes N_{C_2}^H(\mathrm{inf}_e^{C_2}K)\]
Indeed, this follows from from Corollary \ref{cor:nilpotentfromaugideal} by setting $R=MU^{(\!(G)\!)}$, $x=x_i$, and 
\[M=BP^{(\!(G)\!)}/(G\cdot \overline{t}_i^G | i>h)\otimes \mathrm{inf}_{H}^GN_{C_2}^H(\mathrm{inf}_e^{C_2}K),\]
so that 
\[\bpgm{m}\otimes \mathrm{inf}_{H}^GN_{C_2}^H(\mathrm{inf}_e^{C_2}K)=M/(G\cdot\overline{t}_{m+1}^G,\ldots,G\cdot\overline{t}_{h}^G),\]
and
\[\mathrm{Res}^G_HM/(G\cdot\overline{t}_{m+1}^G,\ldots,G\cdot\overline{t}_{h}^G)=\mathrm{Res}^G_H(BP^{(\!(G)\!)}\ip{m})\otimes N_{C_2}^H(\mathrm{inf}_e^{C_2}K).\]

We now consider the subcategory $\mathcal{T}_i\subset Sp^{H,\omega}_{(2)}$ consisting of those finite $2$-local $H$-spectra $X$ with the property that the class $N_{C_2}^H(\overline{t}_i^G)$ acts nilpotently on \[\mathrm{Res}^G_H\bpgm{m}\otimes X\]
in $\mathrm{Mod}_{\Sp^H}(\mathrm{Res}^G_HMU^{(\!(G)\!)})$. We observe that $\mathcal{T}_i$ is a thick tensor-ideal of $Sp^{H,\omega}_{(2)}$ and that  $N_{C_2}^H(\mathrm{inf}_e^{C_2}K)\in\mathcal{T}_i$ for all $1\le i\le m$. The two finite $H$-spectra $N_{C_2}^H(\mathrm{inf}_e^{C_2}K)$ and $\mathrm{inf}_e^H(K)$ are both of chromatic type $m|G|/2+1$ at each subgroup of $H$, in the sense of \cite{balmersanders}. By the determination of the topology of the Balmer spectrum in \cite{6author}, they therefore generate the same thick tensor-ideal, and we conclude that $\mathrm{inf}_e^H(K)\in \mathcal{T}_i$ for all $1\le i\le m$, proving the claim of the proposition.
\end{proof}

\begin{theorem}\label{thm:fptypem|G|/2}
$BP^{(\!(G)\!)}\ip{m}^H$ is fp of type $\le h$ for all $H\subset G$.
\end{theorem}
\begin{proof}
Fixing a type $h+1$ finite ring spectrum $K$ as above, it follows from Lemma \ref{lemma:nilpotentactiononbpgm} and Proposition \ref{prop:ticlasses} that
\[\mathrm{Res}^G_H(BP^{(\!(G)\!)}\ip{m}\otimes \mathrm{inf}_e^GK)\in \mathrm{Thick}_{L\subset H}(H/L_+\otimes H\underline{\Z}\otimes \mathrm{inf}_e^HK).\]
It follows then that 
\begin{align*}
  BP^{(\!(G)\!)}\ip{m}^H\otimes K&=(\mathrm{Res}^G_HBP^{(\!(G)\!)}\ip{m}\otimes \mathrm{inf}_e^HK)^H\\
  &\in\mathrm{Thick}(H\Z\otimes K)\\
  &\subset\mathrm{Thick}(H\F_2)
\end{align*}
where we use the fact that $H\Z$ is an \emph{fp} spectrum of type $\le h$.
\end{proof}

\section{$K$-theory and Euler Characteristics}\label{sec:moorespectra} For each $H\subset G$, the fixed point spectrum $\bpgm{m}^H$ now determines an object in the stable $\infty$-category $\fp{h}$ of \emph{fp} spectra of type $\le h$. The filtration of Section \ref{sec:filtration} allows us to give explicit relations between the $K$-theory classes of these spectra in $K_0(\fp{h})$. 

\subsection{K theory relations for fixed point spectra}

The regular representation sphere $S^{\rho_G}$ for a cyclic $2$-group $G$ has a straightforward cell structure. This allows us to determine the $K$-theory class of $RO(G)$ shifts of $X$ in terms of that of $X$. We use here the defining property of $K$-theory, which states that the $K$-theory class of an object with a finite filtration is the same as the $K$-theory of its associated graded. 

\begin{notation}
    In this section, for a set of spectra $T$, let $\langle T\rangle$ denote the smallest stable subcategory of spectra containing $T$.
\end{notation}

\begin{proposition}\label{prop:ktheoryofshift}
    Let $G=C_{2^n}$, $G'=C_{2^{n-1}}$ and $X\in \Sp^G$, then one has $K$-theory relations
    \[[(\Sigma^{m\rho_G}X)^G]=\begin{cases}
        [X^{G'}]-[X^G]&m=2k+1\\
        [X^G]&m=2k
    \end{cases}\]
    in $K_0(\langle X^H\rangle_{H\subset G})$
\end{proposition}
\begin{proof}
    The cell structure on $S^{\rho_G}$ determines a filtration
    \[
    \begin{tikzcd}
        S^1\arrow[r]&K_2\arrow[r]\arrow[d]&K_3\arrow[d]\arrow[r]&\cdots\arrow[r]&K_{2^n}=S^{\rho_G}\arrow[d]\\
        &L_2&L_3&&L_{2^n}
    \end{tikzcd}
    \]
    where $L_j=\mathrm{Ind}_{C_{2^{n-i}}}^{C_{2^n}}S^j$ for $2^{i-1}<j\le 2^{i}$, by \cite[Proposition 8.4.7]{HHRbook}. This gives a $K$-theory relation
    \[[(\Sigma^{\rho_G}X)^G]=-[X^G]+\sum\limits_{i=1}^n\sum\limits_{j=2^{i-1}+1}^{2^{i}}(-1)^j[X^{C_{2^{n-i}}}]\]
    The terms in the inner summation cancel in pairs except when $i=1$, leaving $[X^{G'}]$. The statement of the proposition then follows by induction on $m$.
\end{proof}

While an ordinary cofiber $M/x$ is in the stable subcategory $\langle M\rangle$, the cellular filtration of $S^{\rho_G}$ along with the filtration of Proposition \ref{prop:associatedgradedformula} immediately imply the following for equivariant cofibers $M/(G\cdot \bar{x})$.

\begin{proposition}\label{prop:equivariantcofiberstablesubcat}
    Let $M$ be an $\mathbb{S}[G\cdot \bar{x}]$-module for $|x|=k\rho_2$. For all $H\subset G$, the equivariant cofiber $M/(G\cdot\bar{x})^H$is in the stable subcategory $\langle M^H\rangle_{H\subset G}$.
\end{proposition}

We may now apply Proposition \ref{prop:associatedgradedformula} to compute the $K$-theory class $[M]$ in terms of $[M/(G\cdot x)]$. As in the notation of Proposition \ref{prop:associatedgradedformula}, we let $M$ be an $\mathbb{S}[G\cdot x]$-module with $|x|=k\rho_2$, and we ask here that $k$ be odd. The following relation is immediate from Propositions \ref{prop:associatedgradedformula} and \ref{prop:ktheoryofshift}.

\begin{corollary}\label{cor:ktheoryoffiltration}
    There is a $K$-theory relation
    \begin{align*}
        2[M^{G}]&=[M^{G'}]+[M/(G\cdot \bar{x})^{G}]\\
        &+\sum_{\substack{f\in\Map^{C_2}\big(G,\{0,1\}\big)/G \\\mathrm{nonconstant }}}\bigg[\Sigma^{n_f\cdot k\rho_{H_f}}\frac{M}{(H_f\cdot g\bar{x}\mid gH_f\in f^{-1}(0)/H_f)}^{H_f}\bigg]
    \end{align*}
    in $K_0(\langle M^H\rangle_{H\subset G})$.
\end{corollary}

\begin{example}
    With notation as above, setting $G=C_4$, there is a $K$-theory relation
    \begin{align*}
       2[M^{C_4}]&=[M^{C_2}]+[M/(C_4\cdot x)^{C_4}]\\
       &+[M/\bar{x}^e]-[M/\bar{x}^{C_2}] 
    \end{align*}
in $K_0(\langle M^H\rangle_{H\subset G})$.
\end{example}

\begin{example}
    With notation as above, setting $G=C_8$, there is a $K$-theory relation
    \begin{align*}
        2[M^{C_8}]&=[M^{C_4}]+[M/(C_8\cdot \bar{x})^{C_8}]\\
        &+[M/(\bar{x},\gamma \bar{x},\gamma^2 \bar{x})^e]-[M/(\bar{x},\gamma \bar{x},\gamma^2 \bar{x})^{C_2}]\\
        &+[M/(\bar{x},\gamma^2 \bar{x})^{C_2}]-[M/(C_4\cdot \bar{x})^{C_4}]\\
        &+[M/(\bar{x},\gamma \bar{x})^{C_2}]\\
        &+[M/\bar{x}^e]-[M/\bar{x}^{C_2}]
    \end{align*}
in $K_0(\langle M^H\rangle_{H\subset G})$.
\end{example}

\subsection{Torsion $K$-theory classes} Nonequivariantly, the cofiber sequence
\[\Sigma^{2(2^n-1)}BP\ip{n}\xrightarrow{v_n}BP\ip{n}\to BP\ip{n-1}\]
gives the $K$-theory relation
\[[BP\ip{n-1}]=0\in K_0(\ip{BP\ip{n}})\]
Equivariantly, we will use our filtration to show that the classes $[\bpgm{m-1}^H]$ are \emph{torsion} in $K_0(\ip{\bpgm{m}^H}_{H\subset G})$ for all $H\subset G$. It seems plausible that these $K$-theory classes are in fact zero, but we do not know this, and we won't need it. 

\begin{proposition}\label{prop:ktheoryclassestorsion}
    Let $M$ be an $\mathbb{S}[G\cdot \bar{x}_1,\ldots,G\cdot \bar{x}_n]$-module for $|x_i|=(2k_i+1)\rho_2$. If the classes
    \[[M/(G\cdot \bar{x}_1,\ldots,G\cdot \bar{x}_n)^H]\in K_0(\langle M^H \rangle_{H\subseteq G})\]
    are torsion for all $H\subseteq G$, then the classes 
    \[[M/(G\cdot \bar{x_1},\ldots,G\cdot\bar{x}_i)^H]\in K_0(\langle M^H \rangle_{H\subseteq G})\]
    are torsion for all $H\subseteq G$, for all $1\le i\le n$.
\end{proposition}
\begin{proof}
    We argue by induction on $|G|$. The base case $G=\{e\}$ follows from the fact that the restriction of $\bar{x}_i$ has even degree, and for a cofiber sequence of spectra
    \[\Sigma^{2k}X\to X\to Y,\]
    one has that $[Y]=0\in K_0(\langle X\rangle)$.

Suppose now that $|G|>1$, and note that, for $H\subset G$ a proper subgroup
\[M/(G\cdot \bar{x}_1,\ldots,G\cdot\bar{x}_i)^H=\begin{cases}M/(H\cdot g\bar{x}_1,\ldots,H\cdot g\bar{x}_i\mid gH\in G/H)&|H|>1\\M/(gx_1,\ldots,gx_i\mid gC_2\in G/C_2)
&H=\{e\}\end{cases}\]
so that by induction $M/(G\cdot \bar{x}_1,\ldots,G\cdot\bar{x}_i)^H$ is torsion in $K_0(\langle M^K\rangle_{K\subset H})$ and therefore also in $K_0(\langle M^H\rangle_{H\subset G})$. 

It therefore suffices to prove the claim for $M/(G\cdot \bar{x}_1,\ldots,G\cdot\bar{x}_i)^G$,
and we proceed by downward induction on $i$, the case $i=n$ being the assumption of the proposition. Let $i\ge1$, and assume then that the classes
    \[[M/(G\cdot \bar{x_1},\ldots,G\cdot\bar{x}_{i+1})^H]\in K_0(\langle M^H \rangle_{H\subseteq G})\]
    are torsion for all $H\subset G$. We set $N:=M/(G\cdot \bar{x_1},\ldots,G\cdot\bar{x}_{i})$, and apply Corollary \ref{cor:ktheoryoffiltration} replacing $M$ and $\bar{x}$ with $N$ and $\bar{x}_{i+1}$, respectively. This gives the relation
    \begin{align*}
        2[N^{G}]&=[N^{G'}]+[N/(G\cdot\bar{x}_{i+1})^{G}]\\
        &+\sum_{\substack{f\in\Map^{C_2}\big(G,\{0,1\}\big)/G \\\mathrm{nonconstant }}}\bigg[\Sigma^{n_f\cdot k\rho_{H_f}}\frac{N}{(H_f\cdot g\bar{x}_{i+1}\mid gH_f\in f^{-1}(0)/H_f)}^{H_f}\bigg]   
    \end{align*}
    in $K_0(\langle N^H\rangle_{H\subset G})$, and therefore also in $K_0(\langle M^H\rangle_{H\subset G})$ by Proposition \ref{prop:equivariantcofiberstablesubcat}. We will show that each term on the right hand side is torsion in $K_0(\langle M^H\rangle_{H\subset G})$, completing the induction.
    
    After applying Proposition \ref{prop:ktheoryofshift} to get rid of shifts in the third term, we see that every term in this summation is torsion by the induction hypothesis on $|G|$. The term $[N/(G\cdot\bar{x}_{i+1})^G]$ is torsion by the inductive hypothesis on $i$, so it remains to examine
    \[[N^{G'}]=[M/(G'\cdot\bar{x}_j,G'\cdot\gamma\bar{x}_j\mid 1\le j\le i)^{G'}]\]
    where $\g$ is a generator of $G$. This class is torsion in $K_0(\langle M^H\rangle_{H\subset G'})$ by the inductive hypothesis on $|G|$ and therefore also in  $K_0(\langle M^H\rangle_{H\subset G})$.
\end{proof}

We now specialize to $M=\bpgm{m}$ and obtain the following.

\begin{theorem}\label{thm:mainktheoryrelation}
    For all $H\subset G$, there is an equation
    \[|H|\cdot[BP^{(\!(G)\!)}\langle m\rangle^H]\equiv[BP^{(\!(G)\!)}\langle m\rangle^e]\in K_0(\fp{h})\]
    modulo torsion.
\end{theorem}
\begin{proof}
    Since 
    \[BP^{(\!(G)\!)}\langle m\rangle /(G\cdot \overline{t}_1^G,\ldots, G\cdot \overline{t}_m^G)^H=\Z\]
    for all $H\subset G$, and the cofiber sequence
    \[\Sigma^2 ku\to ku\to\Z\]
    implies that $[\Z]=0\in K_0(\fp{n})$ for all $n\ge1$, applying the previous proposition shows that 
    \[2[BP^{(\!(G)\!)}\langle m\rangle ^{C_{2^k}}]\equiv [BP^{(\!(G)\!)}\langle m\rangle^{C_{2^{k-1}}}]\]
    for all $1\le k\le n$, modulo torsion.
    Applying this inductively yields the result.
\end{proof}

\subsection{Euler Characteristics} Since $\Z$ is torsion free, Levy's Euler characteristic vanishes on torsion $K$-theory classes, and Theorem \ref{thm:mainktheoryrelation} already gives us the relation
\[|G|\cdot\chi_{\bpgm{m}^G}=\chi_{\bpgm{m}^e}.\]
We will in fact relate the Euler characteristics $\chi_{\bpgm{m}^e}$ and $\chi_{BP\ip{h}}$, by comparing their values on generalized Moore spectra. This hinges on the following result of Burklund--Levy.

\begin{theorem}[Burklund--Levy]\label{thm:burklundlevy}
    Let $X,Y\in\fp{h}$. Then $\chi_X=\chi_Y$ if and only if 
    \[\chi_X(\mathbb{S}/(2^{i_0},\ldots,v_n^{i_h}))=\chi_Y(\mathbb{S}/(2^{i_0},\ldots,v_h^{i_h}))\]
    for some generalized Moore spectrum $\mathbb{S}/(2^{i_0},\ldots,v_h^{i_h})$.
\end{theorem}
\begin{proof}
    Burklund--Levy show that the map
    \[\chi_{BP\ip{h}}:K_0(\Sp^\omega_{> h})\to\Z\]
    is an isomorphism rationally \cite{burklundlevy} (see \cite[Theorem 1.5, Proposition 7.8, and Remark 7.9]{levy}). Since $\chi_{BP\ip{h}}$ evaluated on any generalized Moore spectrum of type $h+1$ is nonzero, it follows also that the source is generated rationally by any generalized Moore spectrum. Since $\chi_X$ and $\chi_Y$ land in the torsion free abelian group $\Z$, the claim follows.
\end{proof}

\begin{lemma}\label{lemma:eulercharacteristicbpn}
If $\mathbb{S}/(2^{i_0},\ldots,v_n^{i_h})$ is a generalized Moore spectrum, then 
\[\chi_{BP\ip{h}}(\mathbb{S}/(2^{i_0},\ldots,v_h^{i_h}))=\prod\limits_{j=0}^hi_j\]
\end{lemma}
\begin{proof}
    This follows from the fact that
    \[\pi_*(BP\ip{h}\otimes \mathbb{S}/(2^{i_0},\ldots,v_h^{i_h}))=\Z[v_1,\ldots,v_h]/(2^{i_0},\ldots,v_h^{i_h})\]
    together with the formula of Lemma \ref{lem:introlemma}.
\end{proof}

For $\chi_{\bpgm{m}^e}$, we have the following related formula.

\begin{proposition}\label{prop:keyEulercharacteristicrelation}
    Let $\mathbb{S}/(2^{i_0},\ldots,v_{h}^{i_{h}})$ be a generalized Moore spectrum, where $h=m|G|/2$. Then one has
    \[\chi_{\bpgm{m}^e}(\mathbb{S}/(2^{i_0},\ldots,v_{h}^{i_{h}}))=\prod\limits_{j=0}^{|G|/2-1}\binom{jm}{m}_2\cdot\prod\limits_{k=0}^{h}i_k.\]
    In particular,
    \[\chi_{\bpgm{m}^e}=\prod\limits_{j=0}^{|G|/2-1}\binom{jm}{m}_2\cdot\chi_{BP\ip{h}}\]
\end{proposition}

\begin{proof}
    By regularity, one has
\[\pi_*(\bpgm{m}^e\otimes \mathbb{S}/(2^{i_0},\ldots,v_{h}^{i_{h}}))=\pi_*(\bpgm{m}^e)/(2^{i_0},\ldots,v_{h}^{i_{h}})\]
The formula $\sum_i\log_2|\pi_i(-)|$ is additive under short exact sequences of evenly graded abelian $2$-groups, so it follows from filtering by powers of the $v_i$'s that
\[\chi_{\bpgm{m}^e}(\mathbb{S}/(2^{i_0},\ldots,v_{h}^{i_{h}}))=\bigg(\prod\limits_{k=0}^{h}i_k\bigg)\cdot \dim_{\F_2}\frac{\pi_*^e\bpgm{m}}{(2,v_1,\ldots,v_{h})}\]
and we conclude by applying Corollary \ref{cor:bpgmodd}. The second statement now follows directly from Theorem \ref{thm:burklundlevy}.
\end{proof}




\begin{corollary}\label{cor:fpgem|G|/2}
    $\bpgm{m}^H$ is \emph{fp} type $h$ for all $H\subset G$.
\end{corollary}
\begin{proof}
    By Theorem \ref{thm:fptypem|G|/2}, $BP^{(\!(G)\!)}\ip{m}^H$ is \emph{fp} type $\le h$. To see that $BP^{(\!(G)\!)}\ip{m}^H$ is not \emph{fp} type $<h$, we may assume for the sake of contradiction that there exists a generalized Moore spectrum
\[F:=\mathbb{S}/(2^{i_0},v_1^{i_1},\ldots,v_{h-1}^{i_{h-1}})\]
of chromatic type $h$ such that $BP^{(\!(G)\!)}\ip{m}^H\otimes F$ is $\pi$-finite. We may choose a $v_{h}$-self map $v:=v_{h}^{i_{h}}$ of $F$ that acts by zero on $\pi_*(BP^{(\!(G)\!)}\ip{m}^H\otimes F)$, since the latter is concentrated in finitely many degrees, and $v$ must have nonzero degree. Since $v$ is in even degree, this gives a short exact sequence
\[0\to \pi_*BP^{(\!(G)\!)}\ip{m}^H\otimes F\to \pi_*BP^{(\!(G)\!)}\ip{m}^H\otimes F/v\to \pi_*\Sigma^{2k+1}BP^{(\!(G)\!)}\ip{m}^H\otimes F\to0\]
for some $k$. Passing to Euler characteristics, this gives 
\[\chi_{BP^{(\!(G)\!)}\ip{m}^H}(F/v)=\chi_{BP^{(\!(G)\!)}\ip{m}^H}(F)-\chi_{BP^{(\!(G)\!)}\ip{m}^H}(F)=0\]
From Theorem \ref{thm:mainktheoryrelation}, we have that
\[\chi_{\bpgm{m}^e}(F/v)=|H|\cdot\chi_{BP^{(\!(G)\!)}\ip{m}^H}(F/v)=0\]
which contradicts Proposition \ref{prop:keyEulercharacteristicrelation}.
\end{proof}

\section{Chromatic localizations of $MU^{(\!(G)\!)}$-modules}\label{sec:telescope} In this section, we show that chromatic localizations force $MU^{(\!(G)\!)}$-modules to be Borel-complete. In other words, after applying chromatic localizations, there is no difference between taking genuine fixed points and taking homotopy fixed points of $MU^{(\!(G)\!)}$-modules. We also see that the fixed points of an $MU^{(\!(G)\!)}$-module satisfies the condition of the telescope conjecture.

\begin{proposition}\label{prop:chromaticlocalizations}
    Let $M\in\mathrm{Mod}(MU^{(\!(G)\!)})$. For all $h\ge0$, for all \[E\in\{K(h),T(h),L_h^f\mathbb{S},L_h\mathbb{S}\}\]
    the map
    \[\mathrm{inf}_e^G(E)\otimes M\to\mathrm{inf}_e^G(E)\otimes F(EG_+,M)\]
    is am equivalence of $G$-spectra. Moreover, the map $\mathrm{inf}_e^GL_h^f\mathbb{S}\otimes M\to\mathrm{inf}_e^GL_h\mathbb{S}\otimes M$ is an equivalence of $G$-spectra.
\end{proposition}
\begin{proof}
It suffices to check that each of the maps induces an equivalence on geometric fixed points $\Phi^H$ for all $H\subset G$.

    For each $E\in\{K(h),T(h),L_h^f\mathbb{S},L_h\mathbb{S}\}$, the map
    \[\mathrm{inf}_e^G(E)\otimes M\to\mathrm{inf}_e^G(E)\otimes F(EG_+,M)\]
    induces an equivalence of underlying spectra since the functor $F(EG_+,-)$ restricts to the identity functor on underlying spectra. The map $\mathrm{inf}_e^GL_h^f\mathbb{S}\otimes M\to\mathrm{inf}_e^GL_h\mathbb{S}\otimes M$ induces an equivalence on underlying spectra because $\mathrm{Res}^G_eM$ is an $MU$-module, and $MU$-modules satisfy the telescope conjecture at all heights.

    Applying $\Phi^H$ for $H\neq \{e\}$, one has that $\Phi^HM$ and $\Phi^HF(EG_+,M)$ are modules over $\Phi^HMU^{(\!(G)\!)}$. The ring spectrum $\Phi^HMU^{(\!(G)\!)}$ is an algebra over $\Phi^HMU^{(\!(H)\!)}=MO$, which is an $H\F_2$-algebra. Since each of the spectra in $\{K(h),T(h),L_h^f\mathbb{S},L_h\mathbb{S}\}$ are $H\F_2$-acyclic, applying $\Phi^H$ to any of the maps in question gives a map with both source and target equal to $0$, and therefore gives an equivalence.
\end{proof}

\begin{corollary}\label{cor:chromaticlocalizations}
  Let $M\in\mathrm{Mod}(MU^{(\!(G)\!)})$. For all $n\ge0$, $H\subset G$, and for all $L\in\{L_{K(h)},L_{T(h)},L_h^f,L_h\}$, the map
    \[L(M^H)\to L(M^{hH})\]
    is an equivalence of spectra. Moreover the map $L_h^f(M^H)\to L_h(M^H)$ is an equivalence of spectra.
\end{corollary}
\begin{proof}
    This is an immediate corollary of the previous proposition by use of the projection formula $(\mathrm{inf}_e^G(X)\otimes Y)^G\simeq X\otimes Y^G$, for $X\in\Sp$ and $Y\in\Sp^G$.
\end{proof}

\begin{theorem}
    The map $\bpgm{m}\to F(EG_+,\bpgm{m})$ induces an equivalence on $\underline{\pi}_i$ for all $i$ sufficiently large.
\end{theorem}
\begin{proof}
    Consider the following commutative square of $G$-spectra
   \[
    \begin{tikzcd}
        \bpgm{m}\arrow[d]\arrow[r]&\mathrm{inf}_e^GL_h^f\mathbb{S}\otimes\bpgm{m}\arrow[d]\\
        F(EG_+,\bpgm{m})\arrow[r]&\mathrm{inf}_e^GL_h^f\mathbb{S}\otimes F(EG_+,\bpgm{m})
    \end{tikzcd}
    \]
    The top map induces an isomorphism on $\underline{\pi}_i$ for all $i$ sufficiently large by \cite[Theorem 3.1.3]{hahnwilson} because $\bpgm{m}^H$ is an \emph{fp} spectrum of type $h$ for all $H$. The righthand map is an equivalence by Proposition \ref{prop:chromaticlocalizations}, so it suffices to show the bottom map induces an equivalence on $\underline{\pi}_i$ for all $i$ sufficiently large.

    We claim first that the bottom map may be identified with the map
    \[F(EG_+,\bpgm{m})\to F(EG_+,\mathrm{inf}_e^GL_h^f\mathbb{S}\otimes\bpgm{m})\]
    Indeed this follows from the fact that the canonical maps
\[\mathrm{inf}_e^GL_h^f\mathbb{S}\otimes F(EG_+,\bpgm{m})\to \mathrm{inf}_e^GL_h^f\mathbb{S}\otimes F(EG_+,\mathrm{inf}_e^GL_h^f\mathbb{S}\otimes\bpgm{m}) \]
\[ F(EG_+,\mathrm{inf}_e^GL_h^f\mathbb{S}\otimes\bpgm{m})\to \mathrm{inf}_e^GL_h^f\mathbb{S}\otimes F(EG_+,\mathrm{inf}_e^GL_h^f\mathbb{S}\otimes\bpgm{m})\]
are both equivalences. This may be checked as in the proof of Proposition \ref{prop:chromaticlocalizations}. The maps induce equivalences on underlying spectra, and each of the geometric fixed points $\Phi^H$ is a module over $\Phi^H(MU^{(\!(G)\!)})\otimes L_h^f\mathbb{S}$, which is zero as $\Phi^H(MU^{(\!(G)\!)})$ is an $H\F_2$-algebra.

    It remains to show that, for all $H\subset G$, the map of spectra \[\bpgm{m}^{hH}\to(L_h^f\bpgm{m})^{hH}\]
    induces an isomorphism on $\pi_i$ for $i$ sufficiently large. Since $\bpgm{m}^e$ is \emph{fp} of type $h$, it follows that $\bpgm{m}^e\to L_h^f\bpgm{m}^e$ induces an isomorphism on $\pi_i$ for $i$ sufficiently large, so that \[H^s(H;\pi_i\bpgm{m}^e)\to H^s(H;\pi_iL_h^f\bpgm{m}^e)\]
    is an isomorphism for $i$ sufficiently large, for all $s$. The claim now follows from the convergence of the homotopy fixed point spectral sequence.
\end{proof}

\begin{corollary}
    For all $H\subset G$, the spectrum $\tau_{\ge0}(\bpgm{m}^{hH})$ is an \emph{fp} spectrum of type $h$.
\end{corollary}
\begin{proof}
    By the theorem, the fiber of the map
    \[\bpgm{m}\to \tau_{\ge0}F(EG_+,\bpgm{m})\]
    is contained in the thick subcategory generated by $\{\mathrm{Ind}_H^G(H\underline{\Z})\}_{H\subset G}$, and therefore is an \emph{fp} spectrum of type $0$. The result then follows from the fact that $\fp{h}$ is a thick subcategory.
\end{proof}

\printbibliography

\end{document}